\documentclass[a4paper,reqno]{amsart}
\usepackage[paperheight=11in,paperwidth=8.5in]{geometry}
\usepackage{amsthm,amssymb,amsmath,hyperref}
\usepackage{color, mdframed}
\usepackage{tikz}
\usetikzlibrary{decorations.pathreplacing}
\usepackage[utf8]{inputenc}
\usepackage{graphicx,enumerate}
\usepackage{mathtools,bm}
\usepackage{enumitem}

\makeatletter
\def\l@section{\@tocline{1}{12pt plus2pt}{0pt}{}{\bfseries}}
\def\l@subsection{\@tocline{2}{0pt}{2pc}{2pc}{}}
\makeatother
\makeatletter
\def\subsection{\@startsection{subsection}{2}{\z@}%
	{-3.25ex\@plus -1ex \@minus -.2ex}%
	{1.5ex \@plus .2ex}%
	{\normalfont\bfseries\boldmath}}
\def\subsubsection{\@startsection{subsubsection}{3}%
	\z@{.5\linespacing\@plus.7\linespacing}{-.5em}%
	{\normalfont\bfseries\boldmath}}
\renewcommand\paragraph{\@startsection{paragraph}{4}{\z@}%
	{3.25ex \@plus1ex \@minus.2ex}%
	{-1em}%
	{\normalfont\normalsize\bfseries}}
\makeatother

\theoremstyle{plain}
\newtheorem{thm}{Theorem}[section]

\newtheorem{cor}[thm]{Corollary}
\newtheorem{obs}[thm]{Observation}

\newtheorem{prop}[thm]{Proposition}

\theoremstyle{definition}

\theoremstyle{remark}
\newtheorem{rem}[thm]{Remark}

\theoremstyle{plain}
\newtheorem{conj}[thm]{Conjecture}
\numberwithin{equation}{section}

\theoremstyle{plain} 
\newcommand{\thistheoremname}{}
\newtheorem{genericthm}[thm]{\thistheoremname}

  \newtheorem*{genericthm*}{\thistheoremname}
\newenvironment{namedthm*}[1]
  {\renewcommand{\thistheoremname}{#1}%
   \begin{genericthm*}}
  {\end{genericthm*}}

\newcommand{\ep}{\epsilon}
\newcommand{\al}{\alpha}
\newcommand{\be}{\beta}

\newcommand{\R}{{\mathbb R}}

\newcommand{\N}{{\mathbb N}}

\newcommand{\Z}{{\mathbb Z}}

\newcommand{\dist}{\hbox{ \rm dist}}

\newcommand{\calA}{{\mathcal A}}
\newcommand{\calB}{{\mathcal B}}

\newcommand{\calF}{{\mathcal F}}

\newcommand{\calI}{{\mathcal I}}

\newcommand{\calL}{{\mathcal L}}
\newcommand{\calM}{{\mathcal M }}
\newcommand{\calN}{{\mathcal N}}

\newcommand{\calP}{{\mathcal P}}

\newcommand{\calS}{{\mathcal S}}
\newcommand{\calT}{{\mathcal T}}

\newcommand{\supp}{{\textrm{supp}}}

\makeatletter
\newcommand{\vast}{\bBigg@{4}}
\newcommand{\Vast}{\bBigg@{5}}
\makeatother

\newcommand{\frakB}{{\mathfrak B}}

\newcommand{\frakm}{{\mathfrak m}}
\newcommand{\frakM}{{\mathfrak M}}

\newcommand{\frakT}{{\mathfrak T}}

\def\udot#1{\ifmmode\oalign{$#1$\crcr\hidewidth.\hidewidth
    }\else\oalign{#1\crcr\hidewidth.\hidewidth}\fi}

\def\R{\mathbb{R}}
\def\Z{\mathbb{Z}}

\def\beq{\begin{equation}}
\def\eeq{\end{equation}}
\makeatletter
\newcommand{\doublewidetilde}[1]{{%
  \mathpalette\double@widetilde{#1}%
}}
\newcommand{\double@widetilde}[2]{%
  \sbox\z@{$\m@th#1\widetilde{#2}$}%
  \ht\z@=.9\ht\z@
  \widetilde{\box\z@}%
}
\makeatother

\def\one{\mbox{1\hspace{-4.25pt}\fontsize{12}{14.4}\selectfont\textrm{1}}}

\usepackage{tikz}

\makeatletter
\def\@makefnmark{%
  \leavevmode
  \raise.9ex\hbox{\fontsize\sf@size\z@\normalfont\tiny\@thefnmark}}
\makeatother

\begin{document}
	
\title[Curved multi-linear operators in the quasi Banach range]{On the boundedness of the curved trilinear Hilbert transform and the curved $n-$linear maximal operator in the quasi-Banach regime}
\author{Bingyang Hu and Victor Lie}

\address{Bingyang Hu: Department of Mathematics and Statistics, Auburn University, 211 Parker Hall, Auburn, AL, U.S.A.}%
\email{bzh0108@auburn.edu}

\address{Victor D. Lie: Department of Mathematics, Purdue University, 150 N. University St, W. Lafayette, IN 47907, U.S.A.  and
The ``Simion Stoilow" Institute of Mathematics of the Romanian Academy, Bucharest, RO 70700, P.O. Box 1-764, Romania}%
\email{vlie@purdue.edu}

\begin{abstract}
Let $n\in\N$,  $\vec{\alpha}=(\al_1,\ldots,\al_n)\in (0,\infty)^n$, $\vec{\be}=(\be_1,\ldots,\be_n)\in (\R\setminus\{0\})^n$, $\vec{f}:=(f_1,\ldots, f_n)\in \calS^n(\R)$ and set 
$$H_{n,\vec{\al},\vec{\be}}(\vec{f})(x):=p.v. \int_{\R} f_1(x+\be_1 t^{\al_1})\ldots f_n(x+\be_n t^{\al_n}) \frac{dt}{t}, \quad x \in \R\,,$$
with $\calM_{n,\vec{\al},\vec{\be}}$ being its maximal operator counterpart. Assume that  $1<p_j<\infty$ and $\frac{1}{2}<r<\infty$ satisfy $\sum_{j=1}^n \frac{1}{p_j}=\frac{1}{r}$. 

Under the \emph{non-resonant} assumption that  $\{\al_j\}_{j=1}^n$ are pairwise distinct we show that
$$\|H_{3,\vec{\al},\vec{\be}}(\vec{f})\|_{L^r}\lesssim_{\vec{\al},\vec{\be}}  \prod_{j=1}^3\|f_j\|_{L^{p_j}}\qquad
\textrm{and} \qquad \|\calM_{n,\vec{\al},\vec{\be}}\|_{L^r}\lesssim_{n,\vec{\al},\vec{\be}} \prod_{j=1}^n\|f_j\|_{L^{p_j}} \qquad \forall\:n\geq 2\,.$$ 
(For the maximal operator the boundedness range can be extended to include the endpoint $\infty$ for both $p_j$ and $r$.)
\end{abstract}

\date{\today}


\thanks{}

\maketitle


\section{Introduction}

Let $n\in\N$,  $\vec{\alpha}=(\al_1,\ldots,\al_n)\in (0,\infty)^n$, $\vec{\be}=(\be_1,\ldots,\be_n)\in (\R\setminus\{0\})^n$, $\vec{f}:=(f_1,\ldots, f_n)\in \calS^n(\R)$, and, for $t\in\R$, set $t^{\vec{\al}}:=(t^{\al_1},\ldots, t^{\al_n})$.  Consider the following two central objects:
\begin{itemize}
\item the \emph{$n$-linear Hilbert transform} along the curve $\vec{\be}\,t^{\vec{\al}}$ defined by
\begin{equation} \label{CTHT}
H_{n,\vec{\al},\vec{\be}}(\vec{f})(x):=p.v. \int_{\R} f_1(x+\be_1 t^{\al_1})\ldots f_n(x+\be_n t^{\al_n}) \frac{dt}{t}, \quad x \in \R;
\end{equation}

\item the \emph{$n$-(sub)linear maximal operator} along the curve $\vec{\be}\,t^{\vec{\al}}$ defined by
\begin{equation} \label{CnLMO}
\calM_{n,\vec{\al},\vec{\be}} (\vec{f})(x):=\sup_{\epsilon>0} \frac{1}{2\epsilon} \int_{-\epsilon}^{\epsilon}  \left|f_1\left(x+\be_1 t^{\alpha_1}\right) \cdots f_n \left(x+\be_n t^{\alpha_n}\right) \right| dt, \quad x \in \R\,.
\end{equation} 
\end{itemize}

The aim of this paper is to study the quasi-Banach boundedness properties for both $H_{n,\vec{\al},\vec{\be}}$ and $\calM_{n,\vec{\al},\vec{\be}}$ in the (purely) non-resonant regime---\emph{i.e.} when these operators do not verify any type of modulation invariance relations---which is equivalent with the requirement 
\begin{equation} \label{nonres}
    \{\al_j\}_{j=1}^n\:\textrm{are \emph{pairwise distinct}}.
\end{equation}

\subsection{Statement of the main results} The main results of our paper are:

\begin{thm} \label{mainthm01}
Let $\{\al_j\}_{j=1}^3\subset(0,\infty)$ be pairwise distinct and let $\{\be_j\}_{j=1}^3\subset\R\setminus\{0\}$. Then, for any $1<p_1, p_2, p_3<\infty$ and $\frac{1}{2}<r<\infty$ satisfying $\frac{1}{p_1}+\frac{1}{p_2}+\frac{1}{p_3}=\frac{1}{r}$, one has
\begin{equation} \label{bdquasitri}
\|H_{3,\vec{\al},\vec{\be}}(\vec{f})\|_{L^r}\lesssim_{\vec{\al},\vec{\be}} \|f_1\|_{L^{p_1}}\,\|f_2\|_{L^{p_2}}\,\|f_3\|_{L^{p_3}}\,.
\end{equation}
\end{thm}

\begin{rem} \label{remark1}
The $n\geq 4$ analogue result for the singular operator \eqref{CTHT} will be the topic of a follow-up paper. However, the corresponding result for the maximal counterpart \eqref{CnLMO} in the general case $n\geq 2$ is treated in the present paper:
\end{rem}

\begin{thm} \label{mainthm03}
Let $n\geq 2$ and  $\{\al_j\}_{j=1}^n\subset(0,\infty)$ be pairwise distinct and let $\{\be_j\}_{j=1}^n\subset\R\setminus\{0\}$. Then, for any $1<p_j\leq \infty$ and $\frac{1}{2}<r\leq\infty$ satisfying $\sum_{j=1}^n \frac{1}{p_j}=\frac{1}{r}$, one has
\begin{equation} \label{bdquasimax}
\|\calM_{n,\vec{\al},\vec{\be}}\|_{L^r}\lesssim_{n,\vec{\al},\vec{\be}} \prod_{j=1}^n\|f_j\|_{L^{p_j}}\,. 
\end{equation}
\end{thm}

\begin{rem}
The central contribution of this paper is given by Theorem \ref{mainthm01}. In contrast with the singular version, the treatment of the maximal counterpart for general $n$ in Theorem \ref{mainthm03} is much easier due to the positivity of the kernel in expression \eqref{CnLMO} which, in turn, implies a simple but key monotonicity property--see relation \eqref{monot}. This latter feature can then be paired with interpolation techniques and prior knowledge on the maximal boundedness range for the bilinear maximal operator obtained by taking $n=2$ in \eqref{CnLMO}. Thus, interestingly enough, for proving the boundedness range stated in Theorem \ref{mainthm03} one only needs (implicitly) the existence of a smoothing inequality for $n=2$. A smoothing inequality for $n\geq 3$ is only required if one intends to push the output Lebesgue index $r$ below $1/2$. However, this latter task requires much more than the presence of a higher order smoothing inequality as hinted by the same $r>1/2$ range restriction in Theorem \ref{mainthm01}. For more on this topic see Remark \ref{range} and Observation \ref{max range}.
\end{rem}

Combining the methods in \cite{DD}, \cite{Lie2024} and \cite{GL22} together with a straightforward adaptation of the proof of the above theorem one also has the following extension: 

\begin{cor}\label{th2ext} Fix $d\in\N$, $n\geq 2$ and for $j\in\{1,\ldots, n\}$ let $P_j(t)=\sum_{r=1}^{d}a_{j,r}\, t^{\al_{j,r}}$ be some non-constant fewnomials/generalized polynomials with $\{a_{j,r}\}_{j,r},\,\{\al_{j,r}\}_{j,r}\subset\R$ and $\{\al_{j,r}\}_{r}$ pairwise distinct for each fixed $j$. Also let $A_j$ be the set consisting of the smallest and largest exponents within $\{\al_{j,r}\}_{j=1}^d$ which correspond to non-zero generalized monomials in the definition of $P_j$. Assume that
\begin{equation} \label{nonrezcond}
\textrm{the sets}\:\:\{A_j\}_{j=1}^n\:\:\textrm{are pairwise disjoint.}\footnote{This latter condition is essentially equivalent with avoiding resonances among the input functions thus eliminating the situation when the operator under analysis exhibits modulation invariance features.} 
\end{equation}
Then, there exists $c(d)\in [1, 2d-1]$ such that for any  $1<p_j\leq \infty$ and $\frac{c(d)}{c(d)+1}<r\leq\infty$ satisfying $\sum_{j=1}^n \frac{1}{p_j}=\frac{1}{r}$, the maximal operator
\begin{equation} \label{CnLMOext}
\calM_{n,\vec{P}} (\vec{f})(x):=\sup_{\epsilon>0} \frac{1}{2\epsilon} \int_{-\epsilon}^{\epsilon}  \left|f_1\left(x+P_1(t)\right) \cdots f_n \left(x+P_n(t)\right) \right| dt, \quad x \in \R
\end{equation} 
obeys
\begin{equation} \label{bdquasimaxext}
\|\calM_{n,\vec{P}}\|_{L^r}\lesssim_{n, \vec{P}} \prod_{j=1}^n\|f_j\|_{L^{p_j}}\,. 
\end{equation}
\end{cor}

\begin{obs}\label{th2extext} As it turns out, one can further extend the applicability of Corollary \ref{th2ext} by including zero and hybrid curvature situations. Indeed, relying on \cite{la3}, \cite{GL22} and the proof of Theorem \ref{mainthm03}, in suitable situations one can completely remove the non-resonant condition \eqref{nonrezcond}, with only a possible change in the value of $c(d)$. One such situation is provided by the following example: in \eqref{CnLMOext} we take $P_j(t)=a_{j,1} t$ for $1\leq j<n$ and $P_n(t)=a_{n,1} t\,+\,a_{n,2} t^{\al}$ with $\{a_{j,1}\}_{j=1}^n\subset\R\setminus\{0\}$ pairwise distinct and $\al\in (0,\infty)\setminus\{1,2\}$. Then, with this choice of $\vec{P}$ and for arbitrary $n\geq 2$ we have that $\calM_{n,\vec{P}}$ verifies \eqref{bdquasimaxext} for any  $1<p_j\leq \infty$ satisfying $\sum_{j=1}^n \frac{1}{p_j}=\frac{1}{r}$ and $\frac{2}{3}<r\leq\infty$.
\end{obs}

\begin{rem}\label{depend}  The parameter $c(d)$ in Corollary \ref{th2ext} may be chosen equal to $\textsf{m}$ with the latter defined as the maximal multiplicity of the nonzero roots of expressions of the form $P'_{j_1}(t)-P'_{j_2}(t)$ with $j_1\not=j_2\in \{1,\ldots, n\}$. Thus, in generic situations, one has  $\textsf{m}=1=c(d)$, while in the (exceptional) highly correlative situations when $\textsf{m}$ is maximal possible, one has $\textsf{m}=2d-1=c(d)$.
\end{rem}

\begin{rem}\label{Pel} Theorem \ref{mainthm03}, Corollary \ref{th2ext} and Observation \ref{th2extext} extend an earlier result in \cite{BK25}. In the latter, the authors therein proved the boundedness of \eqref{CnLMO} under the following concurrent restrictions: 

\emph{(i)} all the fewnomials $\{P_j\}_{j=1}^n$ in \eqref{CnLMOext} must be polynomials;

\emph{(ii)} each polynomial is of the form $P_j(t)=\be_j t^{\al_j}$ with $\be_j\in\R\setminus\{0\}$ and $\{\al_j\}_{j=1}^n\subset\N$ pairwise distinct; 

\emph{(iii)} the output Lebesque index in \eqref{bdquasimaxext} satisfies the condition $r\geq 1$.

The nature of the first two restrictions originate in the key (black box) ingredient employed in \cite{BK25}: a local $n$-smoothing inequality proved in \cite{KMPW2024b} (see also \cite{KMPW2024a}). The proof of the latter relies on the so-called PET induction scheme developed by Bergelson and Leibman in \cite{BL96}, a powerful method  which, in the context described within Corollary \ref{th2ext} above, faces two key restrictions: 

\emph{1)} it only applies to curves described by polynomials with distinct degrees;

\emph{2)} it only applies to large physical scales corresponding in our case to $|t|>>1$. 

Thus, item \emph{1)} is responsible for restriction \emph{(i)} above while the second restriction \emph{(ii)} is due to the usage of a rescaling argument that is meant to circumvent the presence of \emph{2)}.  

We end this remark by noticing that in contrast with the  PET induction scheme and the subsequent methods developed in  \cite{KMPW2024a} and \cite{KMPW2024b}, the Rank II LGC-methodology introduced by the authors in \cite{HL23} has the flexibility of providing the desired local smoothing inequalities with no restrictions  of the form \emph{1)} and/or \emph{2)}.
\end{rem}

We end this section by recalling the following

\begin{conj}\label{maxrange} (\cite{HL23}) The maximal\footnote{Up to end-points.} boundedness range for both \eqref{CTHT} and \eqref{CnLMO} in the non-resonant regime is given by
$$
H_{n,\vec{\al}, \vec{\be}},\,\calM_{n,\vec{\al}, \vec{\be}}: L^{p_1}(\R) \times \dots \times L^{p_n}(\R) \longmapsto L^r(\R)\,,
$$
where
\begin{equation} \label{maxrange}
\sum_{j=1}^n \frac{1}{p_j}=\frac{1}{r}\quad\textrm{with}\quad 1<p_j < \infty\quad \textrm{and}\quad \frac{1}{n}<r<\infty\,.
\end{equation}
\end{conj}

\begin{rem}\label{range}
Interestingly enough, the above conjecture is only settled in the cases $n=1$ and $n=2$ and seems highly non-trivial even in the case $n=3$. Indeed, for deciding the latter, it seems likely that one has to employ subtle techniques related to additive combinatorics/incidence geometry. 
\end{rem}

\subsection{A brief historical background and motivation}

The origin, evolution and motivation for the main topic of the present paper have been discussed in great detail in the introductory sections of both  \cite{HL23} and \cite{BHL25}. Consequently, in the interest of brevity, in this section we will only provide a very succinct outline:

As expected, the history of this subject follows a \emph{natural hierarchical complexity} which fundamentally depends on:
\begin{itemize} 
\item the \emph{multi-linear degree} $n$;

\item the \emph{non-zero/hybrid/zero curvature paradigm}.
\end{itemize}
With this in mind we consider three main cases:

\underline{\textsf{Case 1}: $n=1$.} $\quad$ In this situation \eqref{CTHT} corresponds to the classical Hilbert transform whose $L^p$-- boundedness, $1<p<\infty$, was provided by M. Riesz in \cite{MR28} while \eqref{CnLMO} corresponds to the classical Hardy-Littlewood maximal operator whose similar boundedness behavior was settled in \cite{HL30}.

\underline{\textsf{Case 2}: $n=2$.} $\quad$ We split the analysis of this case in three subcases:
\begin{itemize}
\item the \emph{zero curvature} setting corresponds to the situation when in both \eqref{CTHT} and \eqref{CnLMO} one takes $\al_1=\al_2$ and (in order to avoid degeneracy) $\be_1\not=\be_2$; consequently, both operators satisfy a modulation invariance relation. The singular version was motivated by A. Calder\'on's study of the Cauchy transform on Lipschitz curves (\cite{Cal}, \cite{CMM}). However, it took more than two decades until--building on the ground-breaking results of L. Carleson, (\cite{Car66}), and C. Fefferman, (\cite{f})--the first bounds were provided: indeed, in the breakthrough papers \cite{lt1}, \cite{lt2},  M. Lacey and C. Thiele proved the $L^{p_1}\times L^{p_2}\,\longrightarrow\,L^r$ boundedness of \eqref{CTHT} within the range $\frac{1}{p_1}+\frac{1}{p_2}=\frac{1}{r}$, $1<p,q\leq\infty$ and $\frac{2}{3}<r<\infty$. The analogous result for the maximal variant \eqref{CnLMO} was settled in \cite{la3}.
    
\item the \emph{non-zero curvature} setting corresponds to the situation when in both \eqref{CTHT} and \eqref{CnLMO} one takes $\al_1\not=\al_2$, or more generally, when one considers bilinear versions of \eqref{CnLMOext} under the same hypothesis as the one stated in Corollary \ref{th2ext}. The motivation for this research direction is twofold: as curved analogues of the singular/maximal versions mentioned at the item above, and, as continuous analogues of the ergodic theoretical problem regarding the norm and pointwise convergence of Furstenberg non-conventional bilinear averages. This direction was investigated in \cite{Li13}, \cite{LX16} for monomials and polynomials, respectively, and, via different methods, in \cite{Lie15}, \cite{Lie18}, \cite{GL20}, for larger classes of ``non-flat" curves that include generalized polynomials.\footnote{To be more precise, all the enumerated papers treated expressions similar to \eqref{CnLMOext} (or its singular integral analogue) but having the predetermined structure $P_1(t)=t$ and various forms of $P_2(t)$, \emph{e.g.} monomials, polynomials, fewnomials, and such that an analogue form of the non-zero curvature condition \eqref{nonrezcond} remains valid.}

\item the \emph{hybrid curvature} setting--as suggested by the name--combines both zero and non-zero curvature features and is motivated by the desire of better understanding the subtle phenomenon that governs the transition from the zero to non-zero curvature. In the context given by \eqref{CnLMOext}, such a situation exemplified by the choice  $P_1(t)=a_{1,1}\, t$ and  $P_2(t)=a_{2,1}\,t\,+\,a_{2,2}\,t^{\al}$ with $a_{1,1}\not=a_{2,1}$ and $\alpha\in(0,\infty)\setminus\{1,2\}$ was treated in \cite{GL22} by employing the Rank I LGC methodology introduced in \cite{Lie2024}. Notice that with this choice one subsumes the cases mentioned at the previous two items.\footnote{In particular we notice that \eqref{nonrezcond} is no longer satisfied since $A_1\cap A_2=\{1\}$.}
\end{itemize}

\underline{\textsf{Case 3}: $n\geq 3$.} $\quad$ Until very recently, irrespective of the curvature regime, nothing was known. In what follows, for brevity, we focus completely on the case $n=3$ and the singular integral formulation \eqref{CTHT}. All of the cases discussed below have deep inter-connections with the areas of time-frequency analysis, additive combinatorics, ergodic theory and number theory--for more an all these the interested reader in invited to consult the extended discussion in the introductory sections of \cite{HL23} and \cite{BHL25}. With these being said, as before, we split our analysis in three subcases: 
\begin{itemize}
\item the \emph{zero curvature} setting corresponds to the situation when $\al_1=\al_2=\al_3$ and (in order to avoid degeneracy) $\{\be_j\}_{j=1}^3$ pairwise distinct. In this context $H_{3,\vec{\al},\vec{\be}}$ represents the infamous trilinear Hilbert transform whose boundedness behavior constitutes one of the major open questions in harmonic analysis. One of the key difficulties is that this operator has both linear and quadratic modulation invariance features. To date, no boundedness range--in the affirmative--is known.
    
\item the \emph{hybrid curvature} setting is represented in this case by the situation when the set formed with $\{\al_j\}_{j=1}^3$ has cardinality two (and with the obvious non-degeneracy condition on  $\{\be_j\}_{j=1}^3$). In this situation $H_{3,\vec{\al},\vec{\be}}$ has only linear modulation invariance and thus contains as a particular case the behavior of the classical (zero-curvature) Bilinear Hilbert transform. The study of a closely related hybrid trilinear Hilbert transform is the subject of a forthcoming paper (\cite{BBL25}). 
    
\item the \emph{non-zero curvature} setting corresponds to the case when $\{\al_j\}_{j=1}^3$ are pairwise distinct. In this situation the boundedness of $H_{3,\vec{\al},\vec{\be}}$ in the Banach regime has been settled in \cite{HL23} via the introduction of the Rank II LGC methodology. Finally, a partial boundedness range extension to the quasi-Banach regime is treated in the present paper. 
\end{itemize}

\subsection{General philosophy and structure of the paper}

We start our section with the following antithetical discussion addressing our two main results:

\begin{itemize}

\item the main \emph{difference} between Theorem \ref{mainthm01} and Theorem \ref{mainthm03} originates in the extra monotonicity property \eqref{monot} of $\calM_{n,\vec{\al},\vec{\be}}$ as opposed to $H_{n,\vec{\al},\vec{\be}}$ and manifests in the fact that Theorem \ref{mainthm01} requires the full strength of a trlinear local smoothing inequality while Theorem \ref{mainthm03} only involves a bilinear such local smoothing inequality hence the general $n\geq 2$ character of statement \eqref{bdquasimax}. As a direct consequence, the proof of \eqref{bdquasitri} is significantly more involved than that of \eqref{bdquasimax};   

\item the main \emph{similarity} between Theorem \ref{mainthm01} and Theorem \ref{mainthm03} is the invariant nature of the range corresponding to the output Lebesgue index, that is $r\geq 1/2$. This is a consequence of the fact that the key obstruction in proving any bound with $r$ below the threshold $1/2$ resides in the necessity of proving a corresponding \emph{single scale} estimate which does not seem tractable with the tools employed in this paper--see Observation \ref{max range}. 
\end{itemize}

As a consequence of the above in what follows we only focus on the proof of Theorem \ref{mainthm01}, and very briefly review the three main ingredients used therein; these are:
\begin{itemize}
\item a trilinear local smoothing inequality proved in \cite{HL23}; 

\item a new \emph{curved trilinear square function} defined in \eqref{trilinsq} whose control in obtained in Proposition \ref{20231227prop01} via 
a careful analysis involving average shifted maximal operators (see \eqref{20231229eq20}) and shifted maximal square functions arguments--see the proof of \eqref{20231231eq01};

\item a restricted weak-type multilinear interpolation analysis which originates in a common body of works thoughtfully discussed in monograph \cite{MS13}.
\end{itemize}

We end this section with a succinct commentary on the structure of our paper. Leaving aside the current introductory section, our paper consists of two main sections:
\begin{itemize}
\item in Section 2 we discuss in detail the proof of Theorem \ref{mainthm01}: this is further divided in seven subsections with the main body of work being delivered in Subsection \ref{MaindiaglargeK}; 

\item in Section 3 we provide the short proof of Theorem \ref{mainthm03}.
\end{itemize}

\subsection*{Acknowledgments}
The authors would like to thank Fred Lin for carefully reading the manuscript and for providing useful comments.
The first author was supported by the Simons Travel grant MPS-TSM-00007213. The second author was supported by the NSF grant DMS-2400238 and by the Simons Travel grant MPS-TSM-00008072.

\section{The curved trilinear Hilbert transform}

In this section we focus on proving Theorem \ref{mainthm01}. 

\subsection{Preliminaries} 

We start our discussion with the following

\begin{rem} One consequence of the robustness of the LGC--methodology is that the main elements of our present approach are insensitive to the specific choices of the triples
\begin{itemize}
\item $\vec{\al}=(\al_1, \al_2, \al_3)$---as long as the key \emph{non-resonant} condition \ref{nonres} is satisfied;

\item $\vec{\be}=(\be_1, \be_2, \be_3)$ (assuming that all $\be_j$'s are nonzero).
\end{itemize}  
  Thus, in what follows, for expository reasons, we make the choice $\al_j=j$. Moreover, since the present work builds upon the main result in \cite{HL23}, for reader's convenience we adopt the notation therein by taking $\vec{\be}=(-1,\,1,\,1)$.
\end{rem}

With these settled, we start by recalling the spatial-frequency decomposition of the operator $H_{3,\vec{\al},\vec{\be}}(\vec{f})$ as described in \cite[Section 2]{HL23}. For this we consider its quadrilinear  associated form given by
\begin{equation} \label{20240215eq01}
\Lambda(\vec{f}):=\Lambda(f_1, f_2, f_3, f_4):=\int_{\R} \int_{\R} f_1(x-t)f_2(x+t^2)f_3(x+t^3)f_4(x) \frac{dt dx}{t}. 
\end{equation} 

Next, we split the above into smaller pieces which have better time-frequency localization properties:
$$
\Lambda(\vec{f})=\sum_{k \in \Z} \sum_{j, l, m \in \Z} \Lambda_{j, l, m}^k(\vec{f}), 
$$
where $\Lambda_{j, l, m}^k$ has the Fourier multiplier:
$$
\frakm_{j, l, m}^k(\xi, \eta, \tau):= \left[ \int_{\R} e^{i \left(- \xi t+\eta t^2+\tau t^3 \right)} 2^k \rho(2^k t) dt \right] \phi \left(\frac{\xi}{2^{j+k}} \right) \phi \left(\frac{\eta}{2^{l+2k}} \right) \phi \left(\frac{\tau}{2^{m+3k}} \right).
$$
Here, we have 
\begin{enumerate}
    \item [$\bullet$] $\rho \in C_0^\infty(\R)$ is odd with $\supp \rho \subseteq \left[-4, 4 \right] \backslash \left[-\frac{1}{4}, \frac{1}{4} \right]$ and $\sum\limits_{k \in \Z} 2^k \rho(2^k t)=\frac{1}{t}$; 
\item [$\bullet$] $\phi \in C_0^\infty(\R)$ is even with $\supp \phi \subseteq \left[-4, 4 \right] \backslash \left[-\frac{1}{4}, \frac{1}{4} \right]$ and $\sum\limits_{j, l, m \in \Z}  \phi \left(\frac{\xi}{2^{j+k}} \right)  \phi \left(\frac{\eta}{2^{l+2k}} \right)  \phi \left(\frac{\tau}{2^{m+3k}} \right)=1$; \\
\item [$\bullet$] $k \in \Z$ stands for the \emph{spatial} parameter: essentially, we have that $t\in \left[\frac{1}{2^k}, \frac{2}{2^k} \right]$; \\
\item [$\bullet$] $j, l, m \in \Z$ stand for the \emph{frequency} parameters: that is, the input functions for the form $\Lambda_{j, l, m}^k(\vec{f})$ obey $\supp \widehat{f_1} \subseteq \left[2^{j+k}, 2^{j+k+1} \right]$, $\supp \widehat{f_2} \subseteq \left[2^{l+2k}, 2^{l+2k+1} \right]$, $\supp \widehat{f_3} \subseteq \left[2^{m+3k}, 2^{m+3k+1} \right]$, and consequently, $\supp \widehat{f_4} \subseteq \supp \left[2^{j+k}, 2^{j+k+1} \right]+\left[2^{l+2k}, 2^{l+2k+1} \right]+ \left[2^{m+3k}, 2^{m+3k+1}\right]$. 
\end{enumerate}
Further, we decompose $\Lambda_{j, l, m}^k(\vec{f})$ as follows:  
$$
\Lambda=\Lambda^{Hi}+\Lambda^{Lo}=\Lambda_{+}^{Hi}+\Lambda_{-}^{Hi}+\Lambda^{Lo}, \quad \quad \textrm{with} 
$$
\begin{enumerate}
    \item [$\bullet$] the \emph{low oscillatory} component defined by
    $$
\Lambda^{Lo}:=\sum_{k \in \Z} \sum_{(j, l, m) \in \Z^3 \backslash \N^3} \Lambda_{j, l, m}^k;
    $$
    \item [$\bullet$] the \emph{high oscillatory} component defined by 
    $$
    \Lambda^{Hi}:=\sum_{k \in \Z} \sum_{(j, l, m) \in \N^3} \Lambda_{j, l, m}^k, \quad \quad \textrm{and} \quad \quad \Lambda_{\pm}^{Hi}:=\sum_{k \in \Z_{\pm}} \sum_{(j, l, m) \in \N^3} \Lambda_{j, l, m}^k.
    $$
\end{enumerate}
Furthermore, we decompose $\Lambda^{Hi}$ as
\begin{equation} \label{hidec} 
\Lambda^{Hi}=\Lambda^D+\Lambda^{\not D}, \quad \quad \textrm{with} 
\end{equation}
\begin{enumerate}
    \item [$\bullet$] $\Lambda^D$ the \emph{diagonal} component having the frequency parameter range satisfying 
  \begin{equation} \label{diagc}  
    \max\{j, l, m\}-\min\{j, l, m\} \le 300\,;
  \end{equation}  
    \item [$\bullet$] $\Lambda^{\not D}$ the \emph{off-diagonal} component having the frequency parameter range satisfying
   \begin{equation} \label{offdiagc}   
     \max\{j, l, m\}-\min\{j, l, m\}>300\,.
   \end{equation}  
    
\end{enumerate}
In the obvious manner one can extend the above definitions to $\Lambda_{\pm}^D$ and $\Lambda_{\pm}^{\not D}$. 

Once here, we further subdivide the off-diagonal components into two parts:
$$
\Lambda^{\not D}=\Lambda^{\not D, S}+\Lambda^{\not D, NS}, \quad \quad \textrm{with} 
$$
\begin{enumerate}
    \item [$\bullet$] $\Lambda^{\not D, S}$ the \emph{stationary} off-diagonal component addressing the stationary phase regime, which, by \cite[Lemma 2.1]{HL23}, corresponds to the following three situations\footnote{For any $j, l \in \Z$, $j \cong l$ means $|j-l|<20$.}:
    \begin{enumerate}
        \item [(1)] $j \cong l$ and $j>m+100$;
        \item [(2)] $l \cong m$ and $l>j+100$;
        \item [(3)] $m \cong j$ and $m>l+100$.
    \end{enumerate}
    \item [$\bullet$] $\Lambda^{\not D, NS}$ the \emph{non-stationary} off-diagonal component addressing the non-stationary phase regime, which, by \cite[Lemma 2.2]{HL23}, implies that the phase of the multiplier corresponding to the form $\Lambda_{j, l, m}^k$ obeys
    $$
    \left| \frac{d}{dt}  \varphi_{\xi, \eta, \tau}^k\right| \gtrsim 2^j+2^l+2^m, \quad \textrm{where} \ \varphi_{\xi, \eta, \tau}^k(t):= -\frac{\xi}{2^k}t+\frac{\eta}{2^{2k}}t^2+\frac{\tau}{2^{3k}} t^3.
    $$
\end{enumerate}
With these done, we achieve the decomposition
$$
\Lambda=\Lambda_{\pm}^D+\Lambda_{\pm}^{\not D, S}+\Lambda_{\pm}^{\not D, NS}+\Lambda^{Lo}. 
$$
Finally, we summarize all the $L^p$ estimates that have been obtained in \cite{HL23}:
\begin{enumerate}
    \item [(a)] For the low oscillatory component, one has
    $$
    \left| \Lambda^{Lo}(\vec{f}) \right| \lesssim \|\vec{f} \|_{L^{\vec{p}}}:=\prod_{i=1}^4 \left\|f_i\right\|_{L^{p_i}}
    $$
    for $\vec{p}:=(p_1, p_2, p_3, p_4) \in {\bf B}:=\left\{(p_1, p_2, p_3, p_4 ) \big | \sum\limits_{j=1}^4 \frac{1}{p_j}=1 \ \textrm{satisfying}  \ 1<p_1, p_3<\infty, 1<p_2, p_4 \le \infty \right\}$ (see, \cite[Theorem 10.1]{HL23}).
    \item [(b)] For the non-stationary off-diagonal component, one has
    $$
    \left| \Lambda^{\not D, NS}(\vec{f}) \right| \lesssim \| \vec{f} \|_{L^{\vec{p}}} 
    $$
    for $\vec{p} \in \overline{{\bf B}}:=\left\{(p_1, p_2, p_3, p_4) \big | \sum\limits_{j=1}^4 \frac{1}{p_j}=1 \ \textrm{and} \ 1<p_j \le \infty \right\}$
    (see, \cite[Lemma 2.3]{HL23}).
    \item [(c)] The main estimate in \cite{HL23} addresses the stationary component. More precisely, it is shown that there exists $\epsilon>0$ such that for all $(j,l,m)\in \N^3$ one has
    $$
    \left| \Lambda_{j, l, m}^k(\vec{f}) \right |\lesssim 2^{-\epsilon \max\{\min\{2|k|, \max\{j, l, m\} \}, \max\{|j-l|, |l-m|, |j-m| \} \}} \|\vec{f} \|_{L^{\vec{p}}},
    $$
    for $\vec{p} \in {\bf H}^{+}:=\left\{ \vec{p} \in \overline{{\bf B}} \ \big | \ \substack{p_1=2 \ \& \\ (p_2, p_3, p_4) \in \left\{2, \infty \right\}} \right\}$ if $k \ge 0$ and $\vec{p} \in {\bf H}^{-}:=\left\{\vec{p} \in \overline{{\bf B}} \ \big | \ \substack{ p_3=2 \ \& \\  (p_1, p_3, p_4) \in \left\{2, \infty \right\}}  \right\}$ if $k<0$ (see, \cite[(12.3)]{HL23}). 
    
    As a consequence, one obtains
    $$
    \left|\Lambda_{+}^D (\vec{f}) \right|, \left| \Lambda_{+}^{\not D, S} (\vec{f}) \right| \lesssim \|\vec{f}\|_{L^{\vec{p}}}
    \quad \textrm{for} \  \vec{p} \in {\bf B}^{+}:=\left\{ \substack{(p_1, p_2, p_3, p_4) \\ 1<p_i \le \infty, i \in \{1, 2, 3, 4\} \\ p_1 \neq \infty}: \frac{1}{p_1}+\frac{1}{p_2}+\frac{1}{p_3}+\frac{1}{p_4}=1 \right\},
    $$ 
    and 
    $$
    \left|\Lambda_{-}^D (\vec{f}) \right|, \left| \Lambda_{-}^{\not D, S} (\vec{f}) \right| \lesssim \|\vec{f}\|_{L^{\vec{p}}} \quad \textrm{for} \  \vec{p} \in {\bf B}^{-}:=\left\{ \substack{(p_1, p_2, p_3, p_4) \\ 1<p_i \le \infty, i \in \{1, 2, 3, 4\} \\ p_3 \neq \infty}: \frac{1}{p_1}+\frac{1}{p_2}+\frac{1}{p_3}+\frac{1}{p_4}=1 \right\}.
    $$
\end{enumerate}

\subsection{Treatment of the main diagonal (multiscale) component $\Lambda_m^{\ge}(\vec{f}):=\sum_{k \ge m} \Lambda_{m}^k(\vec{f})$.}\label{MaindiaglargeK}

We start this section with few clarifications about the notation that will be used in what follows: recalling \eqref{hidec} and \eqref{diagc} we let

\begin{equation}\label{diaglarge}
\Lambda_{+}^{D}\approx\sum_{k \in \N} \sum_{m \in \N} \Lambda_{m}^k=\sum_{m\in\N}\underbrace{\sum_{0\leq k<m} \Lambda_{m}^k}_{:=\Lambda_m^{\le}}\,+\, \sum_{m\in\N}\underbrace{\sum_{k\geq m} \Lambda_{m}^k}_{:=\Lambda_m^{\ge}}=:\Lambda_+^{D, \le}\,+\,\Lambda_+^{D, \ge}\,.
\end{equation}

Our main goal in this section is to prove the following 

\begin{prop} \label{Mainprop} Fix $m\in\N$ and let $\{F_j\}_{j=1}^{4}$ be any real measurable sets of finite measure. Then, there exists $F_4'\subseteq F_4$ major subset, \emph{i.e.} $|F_4'| \ge \frac{1}{2} |F_4|$, such that for any $|f_j|\leq \chi_{F_j}$ with $1\leq j\leq 3$ and $|f_4|\leq \chi_{F'_4}$ and any $\vec{\theta}=(\theta_1, \theta_2, \theta_3)$ with $0<\theta_j<1$ the following holds:
\begin{equation}\label{diaglargebound}
\left|\Lambda_m^{\ge}(\vec{f})\right|\lesssim_{\vec{\theta}} m^{10}\, |F_1|^{\theta_1}\,|F_2|^{\theta_2}\,|F_3|^{\theta_3}\,|F_4|^{1-\theta_1-\theta_2-\theta_3}\,.
\end{equation}
\end{prop}

\begin{obs} \label{max range} The maximal range stated in Conjecture \ref{maxrange} for the case $n=3$ holds for all the subcomponents of $\Lambda$ but $\Lambda_{\pm}^{D, \leq}$. In particular, we have the following:
\begin{itemize}
\item $\Lambda_{+}^{D, \geq}$ maps boundedly $L^{p_1}(\R) \times L^{p_2}(\R) \times L^{p_3}(\R)\times L^{p_4}(\R)\longmapsto \mathbb{C}$ for any $\sum_{j=1}^4\frac{1}{p_j}=1$ with $1<p_1,\,p_2,\,p_3<\infty$ and hence \eqref{maxrange} holds for this component;

\item the current range restriction in the statement of Theorem \ref{mainthm01}, \textit{i.e.} $r>\frac{1}{2}$, is precisely due to our limited understanding of the behavior of $\Lambda_{+}^{D, \leq}$.
\end{itemize}
Deduce thus that in order to achieve the maximal (up to end-points) conjectural range $r>\frac{1}{3}$ for the full form $\Lambda$ it is enough to prove the analogue of \eqref{diaglargebound} for the \emph{single scale} form $\Lambda_{m}^k$ for any $0\leq k\leq m$. However, this latter task, seems to involve very delicate arguments relying on additive combinatorics/incidence geometry. 
\end{obs}

With these being said we pass now to the proof of Proposition \ref{Mainprop}. Our first step is to properly discretize \eqref{diaglarge}. In order to do so, we start by noticing that
\begin{eqnarray*}
&& \Lambda_m^k (\vec{f} ) =\int_{\R^2} \left(f_1* \check{\phi}_{m+k}\right)(x-t) \left(f_2* \check{\phi}_{m+2k}\right)(x+t^2) \left(f_3*\check{\phi}_{m+3k}\right)(x+t^3) \left(f_4*\check{\phi}_{m+3k}\right)(x) \rho_k(t)dtdx \\
&&= \int_{\R^2} \left(f_1* \check{\phi}_{m+k}\right)\left(x-\frac{t}{2^k} \right) \left(f_2* \check{\phi}_{m+2k}\right)\left(x+\frac{t^2}{2^{2k}} \right) \left(f_3*\check{\phi}_{m+3k}\right)\left(x+\frac{t^3}{2^{3k}} \right) \left(f_4*\check{\phi}_{m+3k}\right)(x) \rho(t)dtdx.
\end{eqnarray*}
By the uncertainty principle, we have for each $i=1, 2, 3$, $f_i * \check{\phi}_{m+ik}$ is morally a constant on intervals of length $2^{-m-ik}$. Therefore, this gives a first discretization of our operator $\Lambda_m^k (\vec{f})$ relative to the $x$-variable:
\begin{eqnarray} \label{20230920eq01}
&& \Lambda_m^k(\vec{f}) \simeq \frac{1}{2^{m+3k}} \sum_{z \in \Z} \int_{\R} \left(f_1* \check{\phi}_{m+k}\right)\left(\frac{z}{2^{m+3k}}-\frac{t}{2^k} \right) \left(f_2* \check{\phi}_{m+2k}\right)\left(\frac{z}{2^{m+3k}}+\frac{t^2}{2^{2k}} \right) \nonumber \\
&& \quad \quad \quad \quad \cdot \left(f_3*\check{\phi}_{m+3k}\right)\left(\frac{z}{2^{m+3k}}+\frac{t^3}{2^{3k}} \right) \left(f_4*\check{\phi}_{m+3k}\right)\left(\frac{z}{2^{m+3k}} \right) \rho(t)dt.
\end{eqnarray}
Next by a second discretization within the $t$-variable, by letting $[1,2]=\bigcup_{p\sim 2^m} \left[\frac{p}{2^{m}}, \frac{p+1}{2^{m}} \right]$ and $z\equiv v$, we have
\begin{eqnarray}\label{bddiagklarge}
&& \Lambda_m^{\ge}(\vec{f})= \left|\sum_{k \ge m} \Lambda_m^k(\vec{f})\right| \lesssim \frac{1}{2^m} \sum_{p \sim 2^m} \sum_{\substack{k \ge m \\ r \in \Z}} \frac{1}{\left|I_r^{m+k} \right|^{\frac{1}{2}}} \left| \left \langle f_1, \Phi_{P_{m+k} \left(r-p \right)} \right \rangle \right| \Bigg[ \sum_{I_u^{m+2k} \subseteq I_r^{m+k}} \frac{1}{\left|I_u^{m+2k} \right|^{\frac{1}{2}}} \nonumber \\ 
&& \quad \quad \quad \quad \quad  \left| \left \langle f_2, \Phi_{P_{m+2k} \left(u+\frac{p^2}{2^m} \right)} \right \rangle \right| \left( \sum_{I_v^{m+3k} \subseteq I_u^{m+2k}} \left| \left \langle f_3, \Phi_{P_{m+3k}\left(v+\frac{p^3}{2^{2m}} \right)} \right \rangle \right| \left | \left \langle f_4, \Phi_{P_{m+3k}(v)} \right \rangle \right| \right) \Bigg ]
\end{eqnarray}
where $I_r^{m+k}:=\left[\frac{r}{2^{m+k}}, \frac{r+1}{2^{m+k}} \right]$, $P_{m+k}(v)$ is the Heisenberg time-frequency tile given by $P_{m+k}(v):=\left[\frac{v}{2^{m+k}}, \frac{v+1}{2^{m+k}} \right] \times \left[2^{m+k}, 2^{m+k+1} \right]$, and $\Phi_{P_{m+k}(v)}$ is the $L^2$-normalized wave packet adapted to it.

Observe that if $I_v^{m+3k} \subseteq I_u^{m+2k}$, then by the assumption $k \ge m$, one has $I_{v+\frac{p^3}{2^{2m}}}^{m+3k} \subseteq 2I_u^{m+2k}$. Therefore, without loss of generality, we may write 
\begin{eqnarray*}
&& \Lambda_m^{\ge}(\vec{f})\lesssim\frac{1}{2^m} \sum_{p \sim 2^m} \sum_{\substack{k \ge m \\ r \in \Z}} \frac{1}{\left|I_r^{m+k} \right|^{\frac{1}{2}}} \left| \left \langle f_1, \Phi_{P_{m+k} \left(r-p \right)} \right \rangle \right| \Bigg[ \sum_{I_u^{m+2k} \subseteq I_r^{m+k}} \frac{1}{\left|I_u^{m+2k} \right|^{\frac{1}{2}}}  \\ 
&& \quad \quad \quad \quad \quad  \left| \left \langle f_2, \Phi_{P_{m+2k} \left(u+\frac{p^2}{2^m} \right)} \right \rangle \right| \left( \sum_{I_v^{m+3k} \subseteq I_u^{m+2k}} \left| \left \langle f_3, \Phi_{P_{m+3k}\left(v \right)} \right \rangle \right| \left | \left \langle f_4, \Phi_{P_{m+3k}\left(v-\frac{p^3}{2^{2m}}\right)} \right \rangle \right| \right) \Bigg ]\\
&& \quad = \int_{\R} \sum_{\substack{k \ge m \\ r \in \Z}} \sum_{\substack{I_u^{m+2k} \subseteq I_r^{m+k} \\ I_v^{m+3k} \subseteq I_u^{m+2k}}} \left( \frac{1}{2^m} \sum_{p \sim 2^m} \frac{\left| \left \langle f_1, \Phi_{P_{m+k} \left(r-p \right)} \right \rangle \right|}{\left|I_r^{m+k} \right|^{\frac{1}{2}}} \frac{\left| \left \langle f_2, \Phi_{P_{m+2k} \left(u+\frac{p^2}{2^m} \right)} \right \rangle \right|}{\left|I_u^{m+2k} \right|^{\frac{1}{2}}} \frac{\left| \left \langle f_4, \Phi_{P_{m+3k} \left(v-\frac{p^3}{2^{2m}} \right)} \right \rangle \right|}{\left|I_v^{m+3k} \right|^{\frac{1}{2}}} \right) \\
&& \quad \quad \quad \cdot \frac{\left|\left \langle f_3, \Phi_{P_{m+3k}(v)} \right \rangle \right|}{ \left|I_v^{m+3k} \right|^{\frac{1}{2}}} \one_{I_v^{m+3k}}(x)dx.
\end{eqnarray*}
By Cauchy-Schwarz, one has
\begin{equation} \label{20231227eq01}
\Lambda_m^{\ge} \left(\vec{f} \right) \lesssim \int_{\R} \calT(f_1, f_2, f_4)(x) \calS f_3(x)dx, 
\end{equation} 
where $\calS$ here is the standard discrete square function and 
$$\calT^2(f_1, f_2, f_4)(x):=$$
\begin{equation}\label{trilinsq} 
\sum_{\substack{k \ge m \\ r \in \Z}} \sum_{\substack{I_u^{m+2k} \subseteq I_r^{m+k} \\ I_v^{m+3k} \subseteq I_u^{m+2k}}} \left( \frac{1}{2^m} \sum_{p \sim 2^m} \frac{\left| \left \langle f_1, \Phi_{P_{m+k} \left(r-p \right)} \right \rangle \right|}{\left|I_r^{m+k} \right|^{\frac{1}{2}}} \frac{\left| \left \langle f_2, \Phi_{P_{m+2k} \left(u+\frac{p^2}{2^m} \right)} \right \rangle \right|}{\left|I_u^{m+2k} \right|^{\frac{1}{2}}} \frac{\left| \left \langle f_4, \Phi_{P_{m+3k} \left(v-\frac{p^3}{2^{2m}} \right)} \right \rangle \right|}{\left|I_v^{m+3k} \right|^{\frac{1}{2}}} \right)^2 \one_{I_v^{m+3k}}(x)\,,
\end{equation}
is a newly introduced \emph{curved trilinear square function} associated to the form $\Lambda$ defined in \eqref{20240215eq01}.

\begin{prop} \label{20231227prop01}
Let $m \in \N$ and $F_1, F_2, F_4$ be measurable sets with finite (nonzero) Lebesgue measure. Then 
\begin{center}
$\exists~F_4' \subseteq F_4$ measurable with $|F_4'| \ge \frac{1}{2} |F_4|$, 
\end{center}
such that for any triple of functions $f_1, f_2$ and $f_4$ obeying
\begin{equation} \label{fj}
|f_1| \le \one_{F_1}, \quad |f_2| \le \one_{F_2}, \quad |f_4| \le \one_{F_4'}, 
\end{equation}
and any $1<p<\infty$, one has that
\begin{equation} \label{20291229eq21}
\left\|\calT(f_1, f_2, f_4) \right\|_{L^p(\R)} \lesssim_p m^{10} \frac{|F_1|}{|F_4|} \frac{|F_2|}{|F_4|} |F_4|^{\frac{1}{p}}.
\end{equation}
\end{prop}

Assume for the moment that Proposition \ref{20231227prop01} holds. Then, for $f_3 \le \one_{F_3}$ with $F_3$ finitely measurable set in $\R$, by \eqref{20231227eq01}, Cauchy-Schwarz and Proposition \ref{20231227prop01}, we have 
$$
\Lambda_m^{\ge} \left(\vec{f} \right) \lesssim \left\|\calT(f_1, f_2, f_4) \right\|_{L^p(\R)} \left\|\calS f_3 \right\|_{L^{p'}(\R)} \lesssim_p m^{10} \frac{|F_1|}{|F_4} \frac{|F_2|}{|F_4|} |F_4|^{\frac{1}{p}} |F_3|^{1-\frac{1}{p}}, \quad \forall\: p \in (1, \infty), 
$$
which immediately implies Proposition \ref{Mainprop}. 

\medskip 

We are thus now left with the proof of Proposition \ref{20231227prop01}; this will be achieved via several steps: 
\medskip 

\noindent\textsf{\underline{Step I}: Level set decomposition.} 
Observe that the local information of the operator $\calT(f_1, f_2, f_4)$ is concentrated in the intervals of the forms $\left\{I_{\widetilde{r}}^k \right\}_{\widetilde{r} \in \Z}$. Moreover, for each $I_{r-p}^{m+k}$, there is a unique $I_{\widetilde{r}}^k$ such that $I_{r-p}^{m+k} \subseteq I_{\widetilde{r}}^k$ and thus it is natural to split $r=2^m \widetilde{r}+ s$ with $s\sim 2^m$.

With this we rewrite 
$$\calT^2(f_1, f_2, f_4)(x)\approx$$
\begin{equation} 
\sum_{\substack{k \ge m \\ \widetilde{r} \in \Z}} \sum_{\substack{I_u^{m+2k} \subseteq I_{2^m \widetilde{r}+ s}^{m+k}\subseteq I_{\widetilde{r}}^k \\ I_v^{m+3k} \subseteq I_u^{m+2k}}} \left( \frac{1}{2^m} \sum_{p \sim 2^m} \frac{\left| \left \langle f_1, \Phi_{P_{m+k} \left(2^m \widetilde{r}-p \right)} \right \rangle \right|}{\left|I_r^{m+k} \right|^{\frac{1}{2}}} \frac{\left| \left \langle f_2, \Phi_{P_{m+2k} \left(u+\frac{(p+s)^2}{2^m} \right)} \right \rangle \right|}{\left|I_u^{m+2k} \right|^{\frac{1}{2}}} \frac{\left| \left \langle f_4, \Phi_{P_{m+3k} \left(v-\frac{(p+s)^3}{2^{2m}} \right)} \right \rangle \right|}{\left|I_v^{m+3k} \right|^{\frac{1}{2}}} \right)^2 \one_{I_v^{m+3k}}(x). \nonumber
\end{equation}

Notice now that one obviously has $\frac{\int_{I_{2^m \widetilde{r}-p}^{m+k}} \left|f_1\right|}{\left|I_{0}^{m+k} \right|} \lesssim 2^m \frac{\int_{3I_{\widetilde{r}}^k} |f_1|}{\left|I_{\widetilde{r}}^k \right|}$.
Therefore, for each $\widetilde{r} \in \Z$, we let
$$
S_{\widetilde{r}}^l(f_1):=\left\{p \sim 2^m: \frac{\int_{I_{2^m \widetilde{r}-p}^{m+k}} \left|f_1\right|}{\left|I_{0}^{m+k} \right|} \simeq 2^{l} \frac{\int_{3 I_{\widetilde{r}}^k} |f_1|}{\left|I_{\widetilde{r}}^k \right|} \right\}, \quad l \in \{1, \dots, m\}, 
$$
and
$$
S_{\widetilde{r}}^0(f_1):=\left\{p \sim 2^m: \frac{\int_{I_{2^m \widetilde{r}-p}^{m+k}} \left|f_1\right|}{\left|I_{0}^{m+k} \right|} \lesssim  \frac{\int_{3 I_{\widetilde{r}}^k} |f_1|}{\left|I_{\widetilde{r}}^k \right|} \right\}
$$
and remark that $\# S_{\widetilde{r}}^l(f_1) \lesssim 2^{m-l}$. With these, via Cauchy-Schwarz, we have
$$\frac{1}{m}\,\calT^2(f_1, f_2, f_4)(x)\lesssim $$
\begin{equation*}
\sum_{\substack{k \ge m \\ \widetilde{r} \in \Z}} \sum_{l=0}^m  \left(\frac{\int_{3I_{\widetilde{r}}^k} |f_1|}{\left|I_{\widetilde{r}}^k \right|} \right)^2 \sum_{\substack{I_{2^m \widetilde{r}+ s}^{m+k} \subseteq I_{\widetilde{r}}^k \\ I_u^{m+2k} \subseteq I_{2^m \widetilde{r}+ s}^{m+k} \\ I_v^{m+3k} \subseteq I_u^{m+2k}}} \left( \frac{1}{2^{m-l}} \sum_{p \in S_{\widetilde{r}}^l(f_1)} \frac{\left| \left \langle f_2, \Phi_{P_{m+2k} \left(u+\frac{(p+s)^2}{2^m} \right)} \right \rangle \right|}{\left|I_u^{m+2k} \right|^{\frac{1}{2}}} \frac{\left| \left \langle f_4, \Phi_{P_{m+3k} \left(v-\frac{(p+s)^3}{2^{2m}} \right)} \right \rangle \right|}{\left|I_v^{m+3k} \right|^{\frac{1}{2}}} \right)^2 \one_{I_v^{m+3k}}(x). 
\end{equation*}

\medskip 

\noindent\textsf{\underline{Step II}: Definition of the exceptional set.} Let
$$
\Omega:=\Omega_1 \cup \Omega_2 \cup \Omega_3. 
$$
Here, the sets $\Omega_i$'s are defined as follows.

\medskip 

\noindent \textit{Definition of $\Omega_1$.} Let 
$$
\Omega_1:= \left\{x \in \R: M\one_{F_1} \ge 100 \frac{|F_1|}{|F_4|} \right\} \cup  \left\{x \in \R: M\one_{F_2} \ge 100 \frac{|F_2|}{|F_4|} \right\}.
$$
It is clear that $|\Omega_1| \le \frac{|F_4|}{50}$. 

\medskip 

\noindent \textit{Definition of $\Omega_2$.} For each fixed $k \ge m$, we let
$$
\widetilde{\calB^k}:=\bigcup\left\{I_{2^m \widetilde{r}+ s}^{m+k} : \frac{\left| F_2 \cap I_{2^m \widetilde{r}+ s}^{m+k} \right|}{\left|F_4 \cap I_{2^m \widetilde{r}+ s}^{m+k}  \right|} \ge \frac{100 |F_2|}{|F_4|} \right\},
$$
$$
\widetilde{\calB^{\textrm{max}}}:=\left\{I: I \in \bigcup_{k \ge m} \widetilde{\calB^k}, \ I \ \textrm{maximal $\&$ dyadic} \right\}. 
$$
With these, let $\Omega_2:=\bigcup\limits_{I \in \widetilde{\calB^{\textrm{max}}}} I \cap F_4$. We claim that $\left|\Omega_2 \right| \le \frac{|F_4|}{50}$. Indeed, since $\widetilde{\calB^{\textrm{max}}}$ is a collection of maximal dyadic intervals, we have for distinct $I, I' \in \widetilde{\calB^{\textrm{max}}}$, $I \cap I'=\emptyset$. Thus, 
$$ 
\left|\Omega_2 \right|= \left| \bigcup_{I \in \widetilde{\calB^{\textrm{max}}}} I \cap F_4 \right| \le \sum_{I \in \widetilde{\calB^{\textrm{max}}}} \left|I \cap F_4 \right| \le \sum_{I \in \widetilde{\calB^{\textrm{max}}}} \frac{\left|F_2 \cap I \right|}{50 |F_2|} \left|F_4 \right| \le \frac{|F_4|}{50},
$$
which gives the desired claim.

\medskip 

\noindent\textit{Definition of $\Omega_3$.} For each $\widetilde{r} \in \Z$ and $l \in \{0, 1, \dots, m\}$, recalling that $k\geq m$, we define 
$$
\calA_{\widetilde{r}}^l(f_2):=\bigcup_{s\sim 2^m}\left\{u: I_u^{m+2k} \subseteq I_{2^m \widetilde{r}+ s}^{m+k}: \frac{1}{2^{m-l}} \sum_{p \in \calS_{\widetilde{r}}^l(f_1)} \frac{\left| \left\langle f_2, \Phi_{P_{m+2k}\left(u+\frac{(p+s)^2}{2^m} \right)} \right\rangle \right|}{\left|I_u^{m+2k} \right|^{\frac{1}{2}}} \le \frac{1000\,m^3 \left|F_2 \cap I_{2^m \widetilde{r}+ s}^{m+k}\right|}{\left|F_4 \cap I_{2^m \widetilde{r}+ s}^{m+k} \right|} \right\},
$$
$$
\left(\calA_{\widetilde{r}}^l(f_2) \right)^c:=\bigcup_{s\sim 2^m}\left\{u: I_u^{m+2k} \subseteq I_{2^m \widetilde{r}+ s}^{m+k}: \frac{1}{2^{m-l}} \sum_{p \in \calS_{\widetilde{r}}^l(f_1)} \frac{\left| \left\langle f_2, \Phi_{P_{m+2k}\left(u+\frac{(p+s)^2}{2^m} \right)} \right\rangle \right|}{\left|I_u^{m+2k} \right|^{\frac{1}{2}}} > \frac{1000\,m^3 \left|F_2 \cap I_{2^m \widetilde{r}+ s}^{m+k} \right|}{\left|F_4 \cap I_{2^m \widetilde{r}+ s}^{m+k}\right|} \right\},
$$
and
$$
\Omega_3:=\bigcup_{l \in \{0, 1, \dots m\}} \bigcup_{\widetilde{r} \in \Z} \bigcup_{u \in \left(\calA_{\widetilde{r}}^l(f_2) \right)^c} I_u^{m+2k}. 
$$
We claim that $\left| \Omega_3 \right| \le \frac{|F_4|}{50}$. To prove this claim, we first recall the definition of the $p-$shifted maximal function
$M^{(p)}f(x):=\sup\limits_{I_u^k\ni x}\frac{\int_{I_{u+p}^k} |f|}{|I_{u}^k|}$ and note that for each fixed $\widetilde{r} \in \Z$, $s\sim 2^m$ and $l \in \{0, 1, \dots, m\}$, the \emph{average shifted maximal operator} 
\begin{equation} \label{20231229eq20}
\frac{1}{2^{m-l}} \sum_{p \in \calS_{\widetilde{r}}^l(f_1)} M^{\left(\frac{(p+s)^2}{2^m} \right)}: L^1 \longmapsto L^{1, \infty}
\end{equation} 
has the operator norm bounded from above by a constant independent of  $\widetilde{r},\, s,\,l$ and of magnitude at most $m^2$ (this is a consequence of the Stein-Weiss inequality for weak $L^1$ norm and \cite[Lemma 4.3]{GL20}). Let
$$
\frakB_{\widetilde{r}}^l(f_2):=\bigcup_{s\sim 2^m} \left\{ x \in I_{2^m \widetilde{r}+ s}^{m+k}: \frac{1}{2^{m-l}} \sum_{p \in \calS_{\widetilde{r}}^l(f_1)} M^{\left(\frac{(p+s)^2}{2^m} \right)} \one_{F_2 \cap I_{2^m \widetilde{r}+ s}^{m+k}}(x) \ge \frac{1000m^3 \left|F_2 \cap I_{2^m \widetilde{r}+ s}^{m+k}\right|}{\left|F_4 \cap I_{2^m \widetilde{r}+ s}^{m+k} \right|} \right\}. 
$$
Note that $\bigcup\limits_{u \in \left(\calA_{\widetilde{r}}^l(f_2) \right)^c} I_u^{m+2k} \subseteq \frakB_{\widetilde{r}}^l(f_2)$. Using \eqref{20231229eq20}, we have 
$$
\left | \Omega_3 \right | \le \sum_{l=0}^m \sum_{\widetilde{r} \in \Z} \left|\bigcup\limits_{u \in \left(\calA_{\widetilde{r}}^l(f_2)\right)^c} I_u^{m+2k} \right| \le   \sum_{l=0}^m \sum_{\widetilde{r} \in \Z} \left| \frakB_{\widetilde{r}}^l(f_2) \right| \le \sum_{l=0}^m \sum_{\widetilde{r} \in \Z} \sum_{s\sim 2^m} \frac{\left|F_4 \cap I_{2^m \widetilde{r}+ s}^{m+k}\right|}{100 m} \le \frac{|F_4|}{50}, 
$$
which gives the desired claim.

\medskip

As a consequence of the above construction, one has $\left|\Omega \right| \le \left|\Omega_1 \right|+\left|\Omega_2 \right|+\left|\Omega_3 \right| \le \frac{|F_4|}{10}$. With these settled, we set $F_4':=F_4 \backslash \Omega$. 

\medskip 

\noindent\textsf{\underline{Step III}: Decomposition of $ \Omega^c $.} We group the intervals $\left\{I_{\widetilde{r}}^k \right\}_{\widetilde{r} \in \Z}$ based on the relative distance to $ \Omega^c $; more precisely, for each $\beta \in \N$ and $k \ge m$, we define
$$
\calI_{\beta, k}:=\left\{I_{\widetilde{r}}^k: 1+\frac{\textrm{dist} \left(I_{\widetilde{r}}^k, \Omega^c \right)}{\left|I_{\widetilde{r}}^k \right|} \simeq 2^{\beta} \right\}.
$$
This allows us to write $\calT(f_1, f_2, f_4)$ as
$$
\calT^2(f_1, f_2, f_4)(x)=m\,\sum_{l=0}^m \sum_{\beta \ge 0} \calT^2_{\beta, l}(f_1, f_2, f_4)(x), 
$$
where 
$$\calT^2_{\beta, l}(f_1, f_2, f_4)(x):=$$ 
\begin{equation*}
\sum_{\substack{k \ge m \\ \widetilde{r} \in \Z\\ I_{\widetilde{r}}^k \in \calI_{\beta, k}}} \left(\frac{\int_{3I_{\widetilde{r}}^k} |f_1|}{\left|I_{\widetilde{r}}^k \right|} \right)^2 \sum_{\substack{I_{2^m \widetilde{r}+ s}^{m+k} \subseteq I_{\widetilde{r}}^k \\ I_u^{m+2k} \subseteq I_{2^m \widetilde{r}+ s}^{m+k} \\ I_v^{m+3k} \subseteq I_u^{m+2k}}} \left( \frac{1}{2^{m-l}} \sum_{p \in S_{\widetilde{r}}^l(f_1)} \frac{\left| \left \langle f_2, \Phi_{P_{m+2k} \left(u+\frac{(p+s)^2}{2^m} \right)} \right \rangle \right|}{\left|I_u^{m+2k} \right|^{\frac{1}{2}}} \frac{\left| \left \langle f_4, \Phi_{P_{m+3k} \left(v-\frac{(p+s)^3}{2^{2m}} \right)} \right \rangle \right|}{\left|I_v^{m+3k} \right|^{\frac{1}{2}}} \right)^2 \one_{I_v^{m+3k}}(x). 
\end{equation*}

Since the main estimate \eqref{20291229eq21} allows a polynomial loss in $m$, we may focus on a single $l$. Let 
$$
\frakT(f_1, f_2, f_4):=\frac{|F_4|}{|F_1|}\frac{|F_4|}{|F_2|} \calT(f_1, f_2, f_4) \quad \textrm{and} \quad \frakT_{\beta, l}(f_1, f_2, f_4):=\frac{|F_4|}{|F_1|}\frac{|F_4|}{|F_2|} \calT_{\beta, l}(f_1, f_2, f_4). 
$$
With the previous notations, fixing a pair $(f_1,f_2)$ and assuming \eqref{fj}, it is now sufficient to prove that for any $0\leq l\leq m$, $\beta\in\N$ and $1<p<\infty$ one has
\begin{equation} \label{20231231eq05}
\|\frakT_{\beta, l}[f_1, f_2](f_4)\|_{L^p}\lesssim_{p} m^6\, 2^{-100\beta}\,\|f_4\|_{L^p}\,.
\end{equation}

\noindent\textsf{\underline{Step IV}: Treatment of $\frakT_{0, l}$ - $L^2$ estimates.} By the assumption, we have $I_{\widetilde{r}}^k \in \calI_{0, k}$, which implies
\begin{equation} \label{20231231eq05}
\frac{\int_{3 I_{\widetilde{r}}^k} |f_1|}{\left|I_{\widetilde{r}}^k \right|} \lesssim \frac{|F_1|}{|F_4|}.
\end{equation} 
Therefore, using Cauchy-Schwarz, we have
$$\left\| \frakT_{0, l}(f_1, f_2, f_4) \right\|_{L^2(\R)}^2$$
$$\lesssim \frac{|F_4|^2}{|F_2|^2} \int_{\R} \sum_{\substack{k \ge m \\ \widetilde{r} \in \Z \\ I_{\widetilde{r}}^k \in \calI_{0, k}}} \sum_{\substack{I_u^{m+2k}\subseteq I_{2^m \widetilde{r}+ s}^{m+k} \subseteq I_{\widetilde{r}}^k \\ u \in \calA_{\widetilde{r}}^{l}(f_2) \\ I_v^{m+3k} \subseteq I_u^{m+2k}}} \left( \frac{1}{2^{m-l}} \sum_{p \in S_{\widetilde{r}}^l(f_1)} \frac{\left| \left \langle f_2, \Phi_{P_{m+2k} \left(u+\frac{(p+s)^2}{2^m} \right)} \right \rangle \right|}{\left|I_u^{m+2k} \right|^{\frac{1}{2}}} \frac{\left| \left \langle f_4, \Phi_{P_{m+3k} \left(v-\frac{(p+s)^3}{2^{2m}} \right)} \right \rangle \right|}{\left|I_v^{m+3k} \right|^{\frac{1}{2}}} \right)^2 \one_{I_v^{m+3k}}(x) dx$$
\begin{eqnarray*}
&&\lesssim \frac{|F_4|^2}{|F_2|^2} \int_{\R}  \sum_{\substack{k \ge m \\ \widetilde{r} \in \Z \\ I_{\widetilde{r}}^k \in \calI_{0, k}}} \sum_{\substack{I_u^{m+2k}\subseteq I_{2^m \widetilde{r}+ s}^{m+k} \subseteq I_{\widetilde{r}}^k \\ u \in \calA_{\widetilde{r}}^{l}(f_2) \\ I_v^{m+3k} \subseteq I_u^{m+2k}}} \left(\frac{1}{2^{m-l}} \sum_{p \in S_{\widetilde{r}}^l(f_1)} \frac{\left | \left \langle f_2, \Phi_{P_{m+2k} \left(u+\frac{(p+s)^2}{2^m} \right)} \right \rangle \right|}{\left|I_u^{m+2k} \right|^{\frac{1}{2}}} \right) \\
&& \quad \quad \quad \quad \quad \quad \quad \quad \quad \cdot \left( \frac{1}{2^{m-l}} \sum_{p \in S_{\widetilde{r}}^l(f_1)} \frac{\left| \left \langle f_2, \Phi_{P_{m+2k} \left(u+\frac{(p+s)^2}{2^m} \right)} \right \rangle \right|}{\left|I_u^{m+2k} \right|^{\frac{1}{2}}} \frac{\left| \left \langle f_4, \Phi_{P_{m+3k} \left(v-\frac{(p+s)^3}{2^{2m}} \right)} \right \rangle \right|^2}{\left|I_v^{m+3k} \right|} \right) \one_{I_v^{m+3k}}(x)dx. 
\end{eqnarray*}
By the definitions of $\Omega_3$ and $\Omega_2$, we see the expression 
$$
\frac{1}{2^{m-l}} \sum_{p \in S_{\widetilde{r}}^l(f_1)} \frac{\left | \left \langle f_2, \Phi_{P_{m+2k} \left(u+\frac{(p+s)^2}{2^m} \right)} \right \rangle \right|}{\left|I_u^{m+2k} \right|^{\frac{1}{2}}} \quad \textrm{is bounded above by} \quad \frac{10^3m^3 \left|F_2 \cap I_{2^m \widetilde{r}+ s}^{m+k} \right|}{\left|F_4 \cap I_{2^m \widetilde{r}+ s}^{m+k} \right|} \le \frac{10^5 m^3 |F_2|}{|F_4|}. 
$$
Therefore, using the fact that $k \ge m$, we have
$$\left\|\frakT_{0, l}(f_1, f_2, f_4) \right\|_{L^2(\R)}^2$$
$$\lesssim \frac{m^3|F_4|}{|F_2|} \int_{\R}  \sum_{\substack{k \ge m \\ \widetilde{r} \in \Z \\ I_{\widetilde{r}}^k \in \calI_{0, k}}} \sum_{\substack{I_u^{m+2k}\subseteq I_{2^m \widetilde{r}+ s}^{m+k} \subseteq I_{\widetilde{r}}^k \\ u \in \calA_{\widetilde{r}}^{l}(f_2) \\ I_v^{m+3k} \subseteq I_u^{m+2k}}}  \frac{1}{2^{m-l}} \sum_{p \in S_{\widetilde{r}}^l(f_1)} \frac{\left| \left \langle f_2, \Phi_{P_{m+2k} \left(u+\frac{(p+s)^2}{2^m} \right)} \right \rangle \right|}{\left|I_u^{m+2k} \right|^{\frac{1}{2}}} \frac{\left| \left \langle f_4, \Phi_{P_{m+3k} \left(v-\frac{(p+s)^3}{2^{2m}} \right)} \right \rangle \right|^2}{\left|I_v^{m+3k} \right|}  \one_{I_v^{m+3k}}(x)dx$$
$$= \frac{m^3|F_4|}{|F_2|}  \sum_{\substack{k \ge m \\ \widetilde{r} \in \Z \\ I_{\widetilde{r}}^k \in \calI_{0, k}}} \sum_{\substack{I_u^{m+2k}\subseteq I_{2^m \widetilde{r}+ s}^{m+k} \subseteq I_{\widetilde{r}}^k \\ u \in \calA_{\widetilde{r}}^{l}(f_2) \\ I_v^{m+3k} \subseteq I_u^{m+2k}}}  \frac{1}{2^{m-l}} \sum_{p \in S_{\widetilde{r}}^l(f_1)} \frac{\left| \left \langle f_2, \Phi_{P_{m+2k} \left(u+\frac{(p+s)^2}{2^m} \right)} \right \rangle \right|}{\left|I_u^{m+2k} \right|^{\frac{1}{2}}} \left| \left \langle f_4, \Phi_{P_{m+3k} \left(v-\frac{(p+s)^3}{2^{2m}} \right)} \right \rangle \right|^2 $$
$$\lesssim \frac{m^3|F_4|}{|F_2|}  \sum_{\substack{k \ge m \\ \widetilde{r} \in \Z \\ I_{\widetilde{r}}^k \in \calI_{0, k}}} \sum_{\substack{I_u^{m+2k}\subseteq I_{2^m \widetilde{r}+ s}^{m+k} \subseteq I_{\widetilde{r}}^k \\ u \in \calA_{\widetilde{r}}^{l}(f_2)}}  \frac{1}{2^{m-l}} \sum_{p \in S_{\widetilde{r}}^l(f_1)} \frac{\left| \left \langle f_2, \Phi_{P_{m+2k} \left(u+\frac{(p+s)^2}{2^m} \right)} \right \rangle \right|}{\left|I_u^{m+2k} \right|^{\frac{1}{2}}} \left( \sum_{I_v^{m+3k} \subseteq 2I_u^{m+2k}}\left| \left \langle f_4, \Phi_{P_{m+3k} \left(v \right)} \right \rangle \right|^2\right)$$
$$\lesssim m^6 \sum_{k \ge m, v \in \Z} \left | \left \langle f_4, \Phi_{P_{m+3k}(v)} \right \rangle \right|^2 \lesssim m^6 \left\|f_4 \right\|_{L^2(\R)}^2. $$
This concludes the $L^2$ estimate of the operator $\frakT_{0, l}(f_1, f_2, f_4)$.

\medskip 

\noindent\textsf{\underline{Step V}: Treatment of $\frakT_{0, l}$ - $L^1$ estimates.} Our goal is to prove that for any $\lambda>0$, one has
\begin{equation} \label{20231231eq01}
\left| \left\{x \in \R: \left |\frakT_{0, l}(f_1, f_2, f_4)(x) \right|>\lambda \right\} \right| \lesssim \frac{m^6}{\lambda} \left\|f_4 \right\|_{L^1(\R)}. 
\end{equation} 
We prove the above estimate via standard Calder\'on-Zygmund analysis. Let $E_\lambda:=\left\{x \in \R: M_d f_4(x)>\lambda \right\}$, where $M_d$ is the standard dyadic maximal operator. This gives a unique collection $\calI$ of maximal (disjoint) dyadic intervals such that $E_\lambda=\bigcup\limits_{J \in \calI} J$. The following facts are standard: 
\begin{enumerate}
    \item [(1).] $\left| E_\lambda \right| \lesssim \frac{1}{\lambda} \left\|f_4 \right\|_{L^1(\R)}$;
    \item [(2).] $\left| f_4(x) \right| \le \lambda, \forall x \in \R \backslash E_\lambda$;
    \item [(3).] $f_4=g+b$, where $g:=f \one_{\left(E_\lambda\right)^c}+\sum\limits_{J \in \calI} \left(\frac{1}{|J|}\int_J f_4 \right) \one_J$, $b:=\sum\limits_{J \in \calI} b_J$ and $b_J:=\left(f_4-\frac{1}{|J|}\int_J f_4 \right) \one_J$; 
    \item [(4).] $\left\|g \right\|_{L^\infty(\R)} \lesssim \lambda$, $\left\|g \right\|_{L^1(\R)} \lesssim \left\|f_4 \right\|_{L^1(\R)}$;
    \item [(5).] For every $J \in \calI$, $\supp \ b_J \subseteq J$, $\int b_J=0$ and $\left\|b_J \right\|_{L^1(\R)} \lesssim \lambda |J|$. 
\end{enumerate}
Therefore, 
\begin{equation} \label{20231231eq02}
\left| \left\{\left|\frakT_{0, l}(f_1, f_2, f_4)(x) \right|> \lambda \right \} \right| \le \left| \left\{ \left| \frakT_{0, l}(f_1, f_2, g)(x) \right|> \frac{\lambda}{2} \right\} \right|+\left| \left\{ \left| \frakT_{0, l}(f_1, f_2, b)(x) \right|> \frac{\lambda}{2} \right\} \right|.
\end{equation}
The first term in the right-hand side of \eqref{20231231eq02} can be estimated via the $L^2$ estimates derived at the previous step:
$$
\left| \left\{ \left| \frakT_{0, l}(f_1, f_2, g)(x) \right|> \frac{\lambda}{2} \right\} \right| \lesssim \frac{1}{\lambda^2} \left\|\frakT_{0, l}(f_1, f_2, g) \right\|_{L^2(\R)}^2 \lesssim \frac{m^6}{\lambda^2} \left\|g \right\|_{L^2(\R)}^2 \lesssim \frac{m^6}{\lambda} \left\|f_4 \right\|_{L^1(\R)}.
$$
Now for the second term in \eqref{20231231eq02}, we further write it as
\begin{eqnarray*} 
\left| \left\{ \left| \frakT_{0, l}(f_1, f_2, b)(x) \right|> \frac{\lambda}{2} \right\} \right|%
&=&\left| \left\{x \in \bigcup_{J \in \calI} 100J:  \left| \frakT_{0, l}(f_1, f_2, b)(x) \right|> \frac{\lambda}{2} \right\} \right| \\
&& \quad \quad +\left| \left\{x \in \left(\bigcup_{J \in \calI} 100J \right)^c:  \left| \frakT_{0, l}(f_1, f_2, b)(x) \right|> \frac{\lambda}{2} \right\} \right| \\
&:=& {\bf A}+{\bf B}. 
\end{eqnarray*} 
The term ${\bf A}$ is easy to estimate:
$$
{\bf A} \lesssim \sum_{J \in \calI} |J| \lesssim \frac{\left\|f_4 \right\|_{L^1(\R)}}{\lambda}. 
$$
While for the term ${\bf B}$, we have 
$$
{\bf B} \lesssim \frac{1}{\lambda}\int_{\left(\bigcup_{J \in \calI} 100J \right)^c} \left| \frakT_{0, l}(f_1, f_2, b) \right| \lesssim \frac{1}{\lambda} \sum_{J \in \calI} \int_{ \left(100J \right)^c} \left|\frakT_{0, l}(f_1, f_2, b_J) \right|. 
$$
Note that it suffices to show that for each $J \in \calI$, 
\begin{equation} \label{20231231eq14}
\int_{ \left(100J \right)^c} \left|\frakT_{0, l}(f_1, f_2, b_J) \right| \lesssim m^4 \lambda |J|. 
\end{equation} 
To see this, first note that by \eqref{20231231eq05}, one has
$$\int_{ \left(100J \right)^c} \left|\frakT_{0, l}(f_1, f_2, b_J) \right| \lesssim \frac{|F_4|}{|F_2|} \int_{\left(100J \right)^c}$$
$$\left( \sum_{\substack{k \ge m \\ \widetilde{r} \in \Z \\ I_{\widetilde{r}}^k \in \calI_{0, k}}} \sum_{\substack{I_u^{m+2k} \subseteq I_{2^m \widetilde{r}+ s}^{m+k} \subseteq I_{\widetilde{r}}^k \\ u \in \calA_{\widetilde{r}}^l(f_2) \\ I_v^{m+3k} \subseteq I_u^{m+2k}}} \left( \frac{1}{2^{m-l}} \sum_{p \in S_{\widetilde{r}}^l(f_1)} \frac{\left| \left\langle f_2, \Phi_{P_{m+2k} \left(u+\frac{(p+s)^2}{2^m} \right)} \right \rangle \right|}{\left|I_u^{m+2k} \right|^{\frac{1}{2}}} \frac{\left| \left\langle b_J, \Phi_{P_{m+3k} \left(v-\frac{(p+s)^3}{2^{2m}} \right)} \right \rangle \right|}{\left|I_v^{m+3k} \right|^{\frac{1}{2}}} \right)^2 \one_{I_v^{m+3k}}(x) \right)^{\frac{1}{2}} dx $$
$$\lesssim \frac{|F_4|}{|F_2|} \int_{\left(100J \right)^c}\sum_{\substack{k \ge m \\ \widetilde{r} \in \Z \\ I_{\widetilde{r}}^k \in \calI_{0, k}}} \sum_{\substack{I_u^{m+2k} \subseteq I_{2^m \widetilde{r}+ s}^{m+k} \subseteq I_{\widetilde{r}}^k \\ u \in \calA_{\widetilde{r}}^l(f_2) \\ I_v^{m+3k} \subseteq I_u^{m+2k}}} \frac{1}{2^{m-l}} \sum_{p \in S_{\widetilde{r}}^l(f_1)} \frac{\left| \left\langle f_2, \Phi_{P_{m+2k} \left(u+\frac{(p+s)^2}{2^m} \right)} \right \rangle \right|}{\left|I_u^{m+2k} \right|^{\frac{1}{2}}} \frac{\left| \left\langle b_J, \Phi_{P_{m+3k} \left(v-\frac{(p+s)^3}{2^{2m}} \right)} \right \rangle \right|}{\left|I_v^{m+3k} \right|^{\frac{1}{2}}} \one_{I_v^{m+3k}}(x)dx. $$
A simple calculation yields
\begin{eqnarray} \label{20231231eq11}
&& \frac{\left| \left\langle b_J, \Phi_{P_{m+3k} \left(v-\frac{(p+s)^3}{2^{2m}} \right)} \right \rangle \right|}{\left|I_v^{m+3k} \right|^{\frac{1}{2}}}  \\
&& \quad \lesssim \int_{\R} \min \left\{\frac{2^{m+3k}}{\left|2^{m+3k}(y+c(J))-v+\frac{(p+s)^3}{2^{2m}} \right|^2+1}, \max_{|u| \le 2^{m+3k}|J|} \frac{2^{2(m+3k)}|J|}{\left|2^{m+3k}c(J)-v+\frac{(p+s)^3}{2^{2m}}+u \right|^2+1} \right\} \left|b_{\widetilde{J}}(y) \right| dy \nonumber, 
\end{eqnarray} 
where here we denote $c(J)$ the center of $J$, $\widetilde{J}:=J-c(J)$ and $b_{\widetilde{J}}(y):=b_J(y+c(J))$. 

Guided by \eqref{20231231eq11}, we further decompose the estimate of the expression $ \int_{ \left(100J \right)^c} \left|\frakT_{0, l}(f_1, f_2, b_J) \right|$ into two parts:
\begin{equation} \label{20240102eq01}
 \int_{ \left(100J \right)^c} \left|\frakT_{0, l}(f_1, f_2, b_J) \right| \le {\bf B}_1+{\bf B}_2, 
\end{equation} 
where 
\begin{eqnarray*} 
{\bf B}_1:=\frac{|F_4|}{|F_2|} \int_{\left(100J \right)^c}\sum_{\substack{k \ge m \\ 2^{m+3k}|J| \le 1\\ \widetilde{r} \in \Z \\ I_{\widetilde{r}}^k \in \calI_{0, k}}} \sum_{\substack{I_u^{m+2k} \subseteq I_{2^m \widetilde{r}+ s}^{m+k}\subseteq I_{\widetilde{r}}^k \\ u \in \calA_{\widetilde{r}}^l(f_2) \\ I_v^{m+3k} \subseteq I_u^{m+2k}}} &&\frac{1}{2^{m-l}} \sum_{p \in S_{\widetilde{r}}^l(f_1)} \frac{\left| \left\langle f_2, \Phi_{P_{m+2k} \left(u+\frac{(p+s)^2}{2^m} \right)} \right \rangle \right|}{\left|I_u^{m+2k} \right|^{\frac{1}{2}}} \\
&&\cdot\frac{\left| \left\langle b_J, \Phi_{P_{m+3k} \left(v-\frac{(p+s)^3}{2^{2m}} \right)} \right \rangle \right|}{\left|I_v^{m+3k} \right|^{\frac{1}{2}}} \one_{I_v^{m+3k}}(x)dx,
\end{eqnarray*} 
and
\begin{eqnarray*} 
{\bf B}_2:=\frac{|F_4|}{|F_2|} \int_{\left(100J \right)^c}\sum_{\substack{k \ge m \\ 2^{m+3k}|J| \ge 1\\ \widetilde{r} \in \Z \\ I_{\widetilde{r}}^k \in \calI_{0, k}}} \sum_{\substack{I_u^{m+2k} \subseteq I_{2^m \widetilde{r}+ s}^{m+k}\subseteq I_{\widetilde{r}}^k \\ u \in \calA_{\widetilde{r}}^l(f_2) \\ I_v^{m+3k} \subseteq I_u^{m+2k}}} && \frac{1}{2^{m-l}} \sum_{p \in S_{\widetilde{r}}^l(f_1)} \frac{\left| \left\langle f_2, \Phi_{P_{m+2k} \left(u+\frac{(p+s)^2}{2^m} \right)} \right \rangle \right|}{\left|I_u^{m+2k} \right|^{\frac{1}{2}}}\\
&& \cdot\frac{\left| \left\langle b_J, \Phi_{P_{m+3k} \left(v-\frac{(p+s)^3}{2^{2m}} \right)} \right \rangle \right|}{\left|I_v^{m+3k} \right|^{\frac{1}{2}}} \one_{I_v^{m+3k}}(x)dx.
\end{eqnarray*} 

\medskip 

\noindent\underline{\textit{Estimate of ${\bf B}_1$.}} By \eqref{20231231eq11}, 
\begin{eqnarray*}
&& {\bf B}_1 \lesssim \frac{|F_4|}{|F_2|} \int_{(100J)^c} \sum_{\substack{k \ge m \\ 2^{m+3k}|J| \le 1}} \sum_{\substack{\widetilde{r} \in \Z \\ I_{\widetilde{r}^k} \in \calI_{0, k}}} \sum_{\substack{I_u^{m+2k}\subseteq I_{2^m \widetilde{r}+ s}^{m+k} \subseteq I_{\widetilde{r}}^k \\ u \in \calA_{\widetilde{r}}^{l}(f_2) \\ I_v^{m+3k} \subseteq I_u^{m+2k}}} \frac{1}{2^{m-l}}  \sum_{p \in S_{\widetilde{r}}^l(f_1)}  \frac{\left| \left\langle f_2, \Phi_{P_{m+2k} \left(u+\frac{(p+s)^2}{2^m} \right)} \right \rangle \right|}{\left|I_u^{m+2k} \right|^{\frac{1}{2}}}  \\
&& \quad \quad \quad  \quad \quad \quad \cdot \left(\int_{\R} \max_{|u| \le 2^{m+3k}|J|} \frac{2^{2(m+3k)}|J|}{\left|2^{m+3k}c(J)-v+\frac{(p+s)^3}{2^{2m}}+u \right|^2+1} \left|b_{\widetilde{J}}(y) \right| dy \right) \one_{I_v^{m+3k}}(x)dx \\
&& \quad \lesssim \frac{|F_4|}{|F_2|} \int_{(100J)^c} \sum_{\substack{k \ge m \\ 2^{m+3k}|J| \le 1}} \sum_{\substack{\widetilde{r} \in \Z \\ I_{\widetilde{r}}^k \in \calI_{0, k}}} \sum_{\substack{I_u^{m+2k} \subseteq I_{2^m \widetilde{r}+ s}^{m+k}\subseteq I_{\widetilde{r}}^k \\ u \in \calA_{\widetilde{r}}^{l}(f_2)}} \left(\frac{1}{2^{m-l}}  \sum_{p \in S_{\widetilde{r}}^l(f_1)}  \frac{\left| \left\langle f_2, \Phi_{P_{m+2k} \left(u+\frac{(p+s)^2}{2^m} \right)} \right \rangle \right|}{\left|I_u^{m+2k} \right|^{\frac{1}{2}}} \right) \\
&& \quad \quad \quad  \quad \quad \quad \cdot \sup_{p \in S_{\widetilde{r}}^l(f_1)} \left(\sum_{I_v^{m+3k} \subseteq I_u^{m+2k}} \left( \int_{\R}\frac{2^{2(m+3k)}|J|}{\left|2^{m+3k}c(J)-v+\frac{(p+s)^3}{2^{2m}}\right|^2+1} \left|b_{\widetilde{J}}(y) \right| dy \right) \one_{I_v^{m+3k}}(x)dx\right)\\
&& \quad \lesssim m^3 \lambda |J| \int_{(100J)^c} \sum_{\substack{k \ge m \\ 2^{m+3k}|J| \le 1}} \sum_{\substack{\widetilde{r} \in \Z \\ I_{\widetilde{r}}^k \in \calI_{0, k}\\I_u^{m+2k} \subseteq I_{\widetilde{r}}^k}} \left(\sup_{p\sim 2^m} \left(\sum_{I_v^{m+3k} \subseteq I_u^{m+2k}}  \frac{2^{2(m+3k)}|J|}{\left|2^{m+3k}c(J)-v+\frac{p^3}{2^{2m}} \right|^2+1} \right) \one_{I_v^{m+3k}}(x)\right) dx,
\end{eqnarray*}
where in the last estimate, we have used the facts that $\left\|b_J \right\|_{L^1(\R)} \lesssim \lambda|J|$, $ 2^{m+3k}|J| \le 1$ and the definition of $\Omega_3$. Let now $p (\cdot): \Z \mapsto \left[2^m, 2^{m+1} \right] \cap \Z$ be the measurable function that assumes the supremum in the last term above. Hence, 
$$ 
{\bf B}_1 \lesssim m^3 \lambda |J| \sum_{\substack{k \ge m \\ 2^{m+3k}|J| \le 1}} 2^{m+3k}|J| \sum_{\substack{\widetilde{r} \in \Z \\ I_{\widetilde{r}}^k \in \calI_{0, k}}} \sum_{\substack{I_u^{m+2k} \subseteq I_{\widetilde{r}}^k \\ I_v^{m+3k} \subseteq I_u^{m+2k}}} \int_{(100J)^c} \frac{2^{m+3k}}{\left|2^{m+3k}c(J)-v+\frac{p^3(u)}{2^{2m}} \right|^2+1} \one_{I_v^{m+3k}}(x) dx.
$$
Note that since $x \in I_v^{m+3k}$ we have $\left|2^{m+3k}x-v \right| \lesssim 1$ and therefore
\begin{equation} \label{20231231eq21}
{\bf B}_1 \lesssim m^3 \lambda |J| \sum_{\substack{k \ge m \\ 2^{m+3k}|J| \le 1}} 2^{m+3k}|J| \sum_{\substack{\widetilde{r} \in \Z \\ I_{\widetilde{r}}^k \in \calI_{0, k}}} \sum_{I_u^{m+2k} \subseteq I_{\widetilde{r}}^k} \int_{(100J)^c} \frac{2^{m+3k}}{\left|2^{m+3k} \left(c(J)-x \right)+\frac{p^3(u)}{2^{2m}} \right|^2+1} \one_{I_u^{m+2k}}(x) dx.
\end{equation}
Without loss of generality, we may assume\footnote{Alternatively, one can simply apply a change variable $x'=x-c(J)$, replace $I_{\widetilde{r}}^k$ by $I_{\widetilde{r}}^k-c(J)$ and then reduce matters to the situation $c(J)=0$.} $c(J)=0$. Applying the change of variable $x \to \frac{x}{2^{m+3k}}$, and focusing first on the integral term in \eqref{20231231eq21} we have
$$
 \int_{(100J)^c} \frac{2^{m+3k}}{\left|2^{m+3k}x-\frac{p^3(u)}{2^{2m}} \right|^2+1} \one_{I_u^{m+2k}}(x) dx=\int_{\R \backslash 100 \cdot 2^{m+3k}J} \frac{\one_{I_u^{m+2k}} \left( \frac{x}{2^{m+3k}} \right)}{\left|x-\frac{p^3\left(u \right)}{2^{2m}} \right|^2+1}  dx.
$$
Observe that if $\one_{I_u^{m+2k}} \left( \frac{x}{2^{m+3k}} \right) \neq 0$, then $2^{k} u \le x \le 2^{k} \left(u+1 \right)$. Recall also that $\frac{p^3(u)}{2^{2m}} \sim 2^m$ and $k \ge m$. These facts imply that  $|x| \simeq 2^{k}|u|$ and hence $\left|x+\frac{p^3(u)}{2^{2m}} \right| \simeq 2^{k} \left|u\right|$ whenever $u\notin\{ -1, 0\}$. Thus, 
\begin{equation}
\sum_{u \in \Z}\int_{\R \backslash 100 \cdot 2^{m+3k}J} \frac{\one_{I_u^{m+2k}} \left( \frac{x}{2^{m+3k}} \right) }{\left|x-\frac{p^3\left(u \right)}{2^{2m}} \right|^2+1} dx \lesssim \int_{\R} \frac{1}{\left|x-\frac{p^3(0)}{2^{2m}} \right|^2+1} dx+\int_{\R} \frac{1}{\left|x-\frac{p^3(-1)}{2^{2m}} \right|^2+1} dx+ \sum_{u\notin\{ -1, 0\}} \frac{2^{k}}{\left(2^{k} \left|u \right| \right)^2} \lesssim 1. \nonumber
\end{equation}
Plugging the above estimate back to \eqref{20231231eq21}, we have 
\begin{equation} \label{20240102eq11}
{\bf B}_1 \lesssim m^3 \lambda |J| \sum\limits_{\substack{k \ge m \\ 2^{m+3k}|J| \le 1}} 2^{m+3k} |J|\lesssim m^3 \lambda |J|.
\end{equation} 

\medskip 

\noindent\underline{\textit{Estimate of ${\bf B}_2$.}} Using again \eqref{20231231eq11}, Fubini and $x \in I_v^{m+3k}$, we see that 
\begin{eqnarray*}
&& {\bf B}_2 \lesssim \frac{|F_4|}{|F_2|} \sum_{\substack{k \ge m \\ 2^{m+3k}|J| \ge 1}} \sum_{\substack{\widetilde{r} \in \Z \\ I_{\widetilde{r}}^k \in \calI_{0, k}}} \sum_{\substack{I_u^{m+2k}\subseteq I_{2^m \widetilde{r}+ s}^{m+k} \subseteq I_{\widetilde{r}}^k \\ u \in \calA_{\widetilde{r}}^{l}(f_2) \\ I_v^{m+3k} \subseteq I_u^{m+2k}}} \frac{1}{2^{m-l}}  \sum_{p \in S_{\widetilde{r}}^l(f_1)}  \frac{\left| \left\langle f_2, \Phi_{P_{m+2k} \left(u+\frac{(p+s)^2}{2^m} \right)} \right \rangle \right|}{\left|I_u^{m+2k} \right|^{\frac{1}{2}}}  \\
&&  \quad \quad \quad \quad \cdot \int_{(100J)^c}  \int_{\R} \frac{2^{m+3k}}{\left|2^{m+3k}(y+c(J)-x)+\frac{(p+s)^3}{2^{2m}}\right|^2+1} \left|b_{\widetilde{J}}(y) \right| \one_{I_v^{m+3k}}(x) dy dx. 
\end{eqnarray*}
Assuming again wlog that $c(J)=0$ we apply the change variable $x \to \frac{x}{2^{m+3k}}$ in order to deduce
\begin{eqnarray*} 
&& {\bf B}_2\lesssim \frac{|F_4|}{|F_2|} \sum_{\substack{k \ge m \\ 2^{m+3k}|J| \ge 1}} \sum_{\substack{\widetilde{r} \in \Z \\ I_{\widetilde{r}}^k \in \calI_{0, k}}} \sum_{\substack{I_u^{m+2k} \subseteq I_{2^m \widetilde{r}+ s}^{m+k}\subseteq I_{\widetilde{r}}^k \\ u \in \calA_{\widetilde{r}}^{l}(f_2)}} \frac{1}{2^{m-l}}  \sum_{p \in S_{\widetilde{r}}^l(f_1)}  \frac{\left| \left\langle f_2, \Phi_{P_{m+2k} \left(u+\frac{(p+s)^2}{2^m} \right)} \right \rangle \right|}{\left|I_u^{m+2k} \right|^{\frac{1}{2}}}  \nonumber \\
&& \quad \quad \quad \quad \cdot \sup_{p\sim 2^m} \left(\sum_{I_v^{m+3k} \subseteq I_u^{m+2k}}  \int_{\R \backslash 2^{m+3k+3} J}  \int_{\R} \frac{|b_J(y)|}{\left|2^{m+3k}y-x+\frac{(p+s)^3}{2^{2m}}\right|^2+1} \one_{I_v^{m+3k}}\left(\frac{x}{2^{m+3k}} \right) dy dx\right).  
\end{eqnarray*}
Let again $p(\cdot): \Z \to \left[2^m, 2^{m+1} \right] \cap \Z$ be the measurable function that achieves the supremum in the second expression above.  Using the definitions of $\Omega_3$ and $\Omega_2$ again, we see that 
\begin{eqnarray} \label{20231231eq30}
{\bf B}_2 \lesssim m^3 \sum_{\substack{k \ge m \\ 2^{m+3k}|J| \ge 1}} \sum_{\substack{\widetilde{r} \in \Z \\ I_{\widetilde{r}}^k \in \calI_{0, k}\\\substack{I_u^{m+2k}\subseteq I_{\widetilde{r}}^k }}} \int_{\R \backslash 2^{m+3k+3} J}  \int_{\R} \frac{|b_J(y)|}{\left|2^{m+3k}y-x+\frac{p^3(u)}{2^{2m}}\right|^2+1} \one_{I_u^{m+2k}}\left(\frac{x}{2^{m+3k}} \right) dy dx 
\end{eqnarray}
Observe that since $x \in \R \backslash 2^{m+3k+3}J$ and $y \in J$ (therefore, $2^{m+3k} y \in 2^{m+3k}J$), we have $\left| 2^{m+3k}y-x \right| \gtrsim 2^{m+3k} |J|$ while $\frac{p^3(u)}{2^{2m}} \simeq 2^m$. These facts yield that $\left|2^{m+3k}y-x+\frac{p(u)^3}{2^{2m}} \right|$ is comparable to  $\left|2^{m+3k}y-x\right|$ for all but at most $\log \left|\frac{p(u)^3}{2^{2m}} \right| \simeq m$ many $k's$. Consequently, it is natural to decompose \eqref{20231231eq30} into two parts:
\begin{equation} \label{20240102eq02}
\eqref{20231231eq30} \le {\bf B}_{2, 1}+{\bf B}_{2, 2},
\end{equation} 
where
$$ {\bf B}_{2, 1}:=m^3 \sum_{\substack{k \ge m \\ 1\leq 2^{m+3k}|J| \leq 10\cdot 2^m}} \sum_{\substack{\widetilde{r} \in \Z \\ I_{\widetilde{r}}^k \in \calI_{0, k}\\\substack{I_u^{m+2k}\subseteq I_{\widetilde{r}}^k }}} \int_{\R \backslash 2^{m+3k+3}J}  \int_{\R} \frac{|b_J(y)|}{\left|2^{m+3k}y-x+\frac{p^3(u)}{2^{2m}}\right|^2+1} \one_{I_u^{m+2k}}\left(\frac{x}{2^{m+3k}} \right) dy dx 
$$
and
$$
{\bf B}_{2, 2}:=m^3 \sum_{\substack{k \ge m \\ 2^{m+3k}|J| \ge 10\cdot 2^m}} \sum_{\substack{\widetilde{r} \in \Z \\ I_{\widetilde{r}}^k \in \calI_{0, k}\\\substack{I_u^{m+2k}\subseteq I_{\widetilde{r}}^k }}} \int_{\R \backslash 2^{m+3k+3}J}  \int_{\R} \frac{|b_J(y)|}{\left|2^{m+3k}y-x+\frac{p^3(u)}{2^{2m}}\right|^2+1} \one_{I_u^{m+2k}}\left(\frac{x}{2^{m+3k}} \right) dy dx \,.
$$

\medskip 

\noindent\underline{Estimate of ${\bf B}_{2, 1}$.} Note that for $m$ being fixed, there are at most $\lesssim m$ many $k$'s such that $1 \le 2^{m+3k}|J| \le 10 \cdot 2^m$. Applying the change variable $z=x-2^{m+3k}y$ we have
$$
{\bf B}_{2, 1} \lesssim m^3 \sum_{\substack{k \ge m \\ 1 \le 2^{m+3k}|J| \le 10 \cdot 2^m}} \sum_{\substack{\widetilde{r} \in \Z \\ I_{\widetilde{r}}^k \in \calI_{0, k}\\\substack{I_u^{m+2k}\subseteq I_{\widetilde{r}}^k }}}\int_{\R \backslash 2^{m+3k+2}J} \int_\R \frac{\left|b_J(y) \right|}{\left|z-\frac{p^3\left(u\right)}{2^{2m}} \right|^2+1} \one_{I_u^{m+2k}} \left(\frac{z}{2^{m+3k}}+y \right)dydz. 
$$
Observe that for $y \in J$, $\one_{I_u^{m+2k}} \left(\frac{z}{2^{m+3k}}+y \right) \le \one_{2I_u^{m+2k}} \left( \frac{z}{2^{m+3k}} \right)$. Therefore, 
\begin{eqnarray*}
&& {\bf B}_{2, 1} \lesssim m^3 \lambda |J| \sum_{\substack{k \ge m \\ 1 \le 2^{m+3k}|J| \le 10 \cdot 2^m}} \sum_{\substack{\widetilde{r} \in \Z \\ I_{\widetilde{r}}^k \in \calI_{0, k}\\\substack{I_u^{m+2k}\subseteq I_{\widetilde{r}}^k }}} \int_{\R \backslash 2^{m+3k+2}J} \frac{1}{\left|z-\frac{p^3\left(u\right)}{2^{2m}} \right|^2+1} \one_{2I_u^{m+2k}} \left(\frac{z}{2^{m+3k}} \right)dz. 
\end{eqnarray*}
Note also that if $\one_{2I_u^{m+2k}} \left(\frac{z}{2^{m+3k}} \right) \neq 0$, then $2^{k}(u-1) \le z \le 2^{k} (u+2)$. Thus 
\begin{equation} \label{20240102eq03}
{\bf B}_{2, 1} \lesssim m^4 \lambda|J| \left(4 \int_{\R} \frac{dz}{z^2+1}+\sum_{k \ge m}\sum_{u \notin \{-2, -1, 0, 1\}} \frac{2^{k}}{\left(2^{k} |u|\right)^2} \right) \lesssim m^3\lambda|J|, 
\end{equation} 
which concludes the estimate of ${\bf B}_{2, 1}$.
\medskip 

\noindent\underline{Estimate of ${\bf B}_{2, 2}$. } By assumption, one has $|x-2^{m+3k}y| \ge 2^{m+3k+2}|J| \ge 10\cdot 2^m$, and hence $\left|2^{m+3k}y-x+\frac{p^3\left(u \right)}{2^{2m}} \right| \approx |x|$. Moreover, if $\one_{I_u^{m+2k}}\left(\frac{x}{2^{m+3k}} \right)  \neq 0$, then $2^{k}u\le x \le 2^{k}(u+1)$. Therefore, 
\begin{equation} \label{20240102eq04} 
{\bf B}_{2, 2} \lesssim m^3\lambda|J| \left(\int_{\R} \frac{1}{x^2+1}dx+ \sum_{k \ge m} \sum_{u\notin\{-1, 0\}} \frac{2^{k}}{2^{2k} u^2} \right) \lesssim m^3 \lambda|J|. 
\end{equation} 

\medskip 

Finally, combining \eqref{20240102eq01}, \eqref{20240102eq11}, \eqref{20240102eq02}, \eqref{20240102eq03}, \eqref{20240102eq04}, the desired estimate \eqref{20231231eq14} follows. 

\medskip 

\noindent\textsf{\underline{Step VI}: Treatment of $\frakT_{0, l}$ - $L^p$ estimates.} The case $1<p \le 2$ is just a consequence of real interpolation between the $p=2$ case (\textsf{Step IV}) and $p=1$ case (\textsf{Step V}). The other part of the range, i.e., $2<p<\infty$, follows from duality. This follows a standard application of the \emph{Khintchine's Inequality} (see, e.g., \cite[Theorem 2.5]{MS13}). For reader's convenience we provide here the sketch of the argument. Let 
\begin{enumerate}
\item [$\bullet$] $\left\{\omega_{I_v^{m+3k}} \right\}_{I_v^{m+3k}}$ be a sequence of i.i.d random variables on $[0, 1]$ with each $\omega_j$ taking values in $\{\pm 1\}$ with equal probability (i.e. the Rademacher functions);
\item [$\bullet$] $\left\{H_{I_v^{m+3k}} \right\}_{I_v^{m+3k}}$ be the $L^2$ normalized Haar functions associated to the dyadic intervals $\left\{I_v^{m+3k} \right\}$.
\end{enumerate}
Let $\bm{\ep}=(\ep_j)_j, \bm{\mu}=(\mu_j)_j\in \ell^{\infty}(\Z)$ be two sequences with $\|\bm{\ep}\|_{\ell^\infty(\Z)},\,\|\bm{\mu}\|_{\ell^\infty(\Z)}\leq 1$. Now, for fixed $f_1,\,f_2$ (here we may assume wlog that $f_2\geq 0$) and $\bm{\ep}, \bm{\mu}$ as before, we define the family of linear operators indexed by $t\in [0,1]$  as follows:
\begin{eqnarray*}
&& \calL^{\bm{\ep}, \bm{\mu}}_{l, 0, \omega(t)}(f_1, f_2, f_4)(x):=\frac{|F_4|}{|F_1|}\frac{|F_4|}{|F_2|} \sum_{\substack{k \ge m \\ r \in \Z \\I_{\widetilde{r}}^k \in \calI_{0, k}}} \sum_{\substack{I_{2^m \widetilde{r}+ s}^{m+k} \subseteq I_{\widetilde{r}}^k \\ I_u^{m+2k} \subseteq I_{2^m \widetilde{r}+ s}^{m+k} \\ I_v^{m+3k} \subseteq I_u^{m+2k}}} 
\frac{\int_{I_{\widetilde{r}}^k}f_1}{\left|I_{\widetilde{r}}^k \right|} \\
&& \quad \cdot \left(\frac{1}{2^{m-l}} \sum_{p \in S_{\widetilde{r}}^l(f_1)} \ep_{u+\frac{(p+s)^2}{2^m}}\:\mu_{v-\frac{(p+s)^3}{2^{2m}}}\: \frac{ \left \langle f_2, \Phi_{P_{m+2k}\left(u+\frac{(p+s)^2}{2^m} \right)} \right \rangle }{\left|I_u^{m+2k} \right|^{\frac{1}{2}}}   \left\langle f_4, \Phi_{P_{m+3k} \left(v-\frac{(p+s)^3}{2^{2m}} \right)} \right \rangle\right)\,\omega_{I_v^{m+3k}}(t)\, H_{I_v^{m+3k}}(x). 
\end{eqnarray*}
Note that by Khintchine's inequality, one has
$$
\left\|\frakT_{0, l}(f_1, f_2, f_4) \right\|_{L^p(\R)}^p \lesssim \sup_{\|\bm{\ep}\|_{\ell^\infty(\Z)},\,\|\bm{\mu}\|_{\ell^\infty(\Z)}\leq 1} \int_{\R} \int_0^1 \left| \calL^{\bm{\ep}, \bm{\mu}}_{l, 0, \omega(t)}(f_1, f_2, f_4)(x) \right|^p dtdx. 
$$
The desired $L^p$ estimate of $\frakT_{0, l}(f_1, f_2, f_4)$ for $2<p<\infty$ then follows by showing that for any $\bm{\ep}, \bm{\mu}$ in the unit ball of $\ell^{\infty}(\Z)$ a suitable ``dual" expression $\calL^{\bm{\ep}, \bm{\mu},*}_{l,0, \omega(t)}$ verifies\footnote{In the $t-$pointwise dual form defined by $\left \langle \calL^{\bm{\ep}, \bm{\mu}}_{l, 0, \omega(t)}(f_1, f_2, f_4), g \right \rangle=:\left \langle f_4, \calL^{\bm{\ep}, \bm{\mu},*}_{l, 0, \omega(t)}(f_1, f_2, g) \right \rangle$ one may assume that $g$ is supported on $\supp \ f_4 \subseteq F_4'$, which is a union of dyadic intervals of the forms $I_u^{m+2k}$.} uniformly in $t\in[0,1]$ and $\bm{\ep}, \bm{\mu}$ the following:
$$
\left\|\calL^{\bm{\ep}, \bm{\mu},*}_{l,0, \omega(t)}(f_1, f_2, g) \right\|_{L^2(\R)} \lesssim m^3 \left\|g\right\|_{L^2(\R)} \quad \textrm{and} \quad \left\|\calL^{\bm{\ep}, \bm{\mu},*}_{l, 0, \omega(t)} (f_1, f_2, g) \right\|_{L^{1, \infty}(\R)} \lesssim m^6 \left\|g \right\|_{L^1(\R)}\,.
$$
These relations can be proved in the similar fashion with their analogues treated in \textsf{Step IV} and \textsf{Step V}. We leave further details to the interested reader.

\medskip 

\noindent\textsf{\underline{Step VII}: Treatment of $\frakT_{\beta, l}$.} Recall that our goal in this case is to show that for any $1<p<\infty$, one has
\begin{equation} \label{20240102eq35}
\left\|\frakT_{\beta, l}(f_1, f_2, f_4) \right\|_{L^p(\R)} \lesssim_p m^{6} 2^{-100\beta} \left\|f_4 \right\|_{L^p(\R)}.
\end{equation} 
The estimate \eqref{20240102eq35} is an immediate consequence of interpolation, dualization (for $p>2$) together with  
\begin{enumerate}
\item [(1).] $L^2 \mapsto L^2$  estimate: for any $N \in \N$, 
$$
\left\|\frakT_{\beta, l}(f_1, f_2, f_4) \right\|_{L^2(\R)} \lesssim_{N} m^3 2^{-100\beta N} \left\|f_4 \right\|_{L^2(\R)};
$$
\item [(2).] $L^1 \mapsto L^{1, \infty}$ estimate: 
$$
\left\|\frakT_{\beta, l}(f_1, f_2, f_4) \right\|_{L^{1, \infty}(\R)} \lesssim 2^{\beta} m^6 \left\|f_4 \right\|_{L^1(\R)}.
$$
\end{enumerate}
The above is proved following similar arguments to the ones presented at \textsf{Step IV},  \textsf{Step V} and \textsf{Step VI} at which one now adds the two items addressing the specifics of the present situation:
\begin{itemize}
\item for $I_{\widetilde{r}}^k \in \calI_{\beta, k}$, one has the counterpart of \eqref{20231231eq05}:
\begin{equation} \label{20240103eq01}
\frac{\int_{I_{\widetilde{r}}^k}|f_1|}{\left|I_{\widetilde{r}} \right|} \lesssim 2^{\beta} \frac{|F_1|}{|F_4|}\,;
\end{equation} 
\item for any $N\in \N$, 
\begin{equation} \label{20231027eq01*}
\left\|\calS_{\calI_{\beta}} f_4 \right\|_{L^2(\R)} \lesssim_N 2^{-100\beta N} \left\|f_4 \right\|_{L^2(\R)},  
\end{equation} 
where here $\calI_{\beta}:=\bigcup\limits_{k \ge m} \calI_{\beta, k}$ and for $\calI$ a collection of dyadic intervals one defines
\begin{equation} \label{20231027eq02*}
\calS_\calI f(x):= \left( \sum_{I_v^{m+3k} \in \calI} \frac{\left| \left \langle f, \Phi_{P_{m+3k}(v)} \right \rangle \right|^2}{|I_v^{m+3k}|} \one_{I_v^{m+3k}}(x) \right)^{\frac{1}{2}}\,.
\end{equation} 
\end{itemize}
\hfill\qedsymbol{}

\subsection{Treatment of the main diagonal single scale component $\Lambda_{m}^{k}$, with $0 \le k \le m$.} \label{20231112sec01}

In contrast with Proposition \ref{Mainprop}, our aim in this section is to prove the following 

\begin{prop} \label{Mainpropsinglesc} 
Fix $m\in\N$ and $0 \le k \le m$. Also, as before, let $\{F_j\}_{j=1}^{4}$ be any real measurable sets of finite measure. Then, there exists $F_4'\subseteq F_4$ major subset such that for any $|f_j|\leq \chi_{F_j}$ with $1\leq j\leq 3$ and $|f_4|\leq \chi_{F'_4}$ and any $\vec{\theta}=(\theta_1, \theta_2, \theta_3)$ with $0<\theta_j<1$  and $\theta_1+\theta_2+\theta_3\leq 2$ the following holds:
\begin{equation}\label{diaglargeboundsgsc}
\left|\Lambda_m^{k}(\vec{f})\right|\lesssim_{\vec{\theta}} |F_1|^{\theta_1}\,|F_2|^{\theta_2}\,|F_3|^{\theta_3}\,|F_4|^{1-\theta_1-\theta_2-\theta_3}\,.
\end{equation}
\end{prop}

\begin{proof}

Departing from the single scale component
\begin{eqnarray*}
&& \Lambda_{m}^{k}(\vec{f})=\frac{1}{2^m} \sum_{p \sim 2^m}  \sum_{r \in \Z} \frac{1}{\left|I_r^{m+k} \right|^{\frac{1}{2}}} \left| \left \langle f_1, \Phi_{P_{m+k}\left(r-p\right)} \right \rangle \right|  \cdot   \\
&& \left [ \sum_{I_u^{m+2k} \subseteq I_r^{m+k}}  \frac{1}{\left|I_u^{m+2k} \right|^{\frac{1}{2}}} \left| \left \langle f_2, \Phi_{P_{m+2k}\left(u+\frac{p^2}{2^m} \right)} \right \rangle \right| \left(\sum_{I_v^{m+3k} \subseteq I_u^{m+2k}} \left| \left \langle f_3, \Phi_{P_{m+3k}(v+\frac{p^3}{2^{2m}})} \right \rangle \right| \left | \left \langle f_4, \Phi_{P_{m+3k}(v)} \right \rangle \right | \right) \right],
\end{eqnarray*}
by a further re-parameterization of the spatial intervals at the smaller scales we have
\begin{eqnarray*} \label{20231009eq01}
&&\Lambda_{m}^{k}(\vec{f})=\frac{1}{2^m} \sum_{p\sim 2^m}  \sum_{r \in \Z} \sum_{a, b \sim 2^k} \frac{1}{\left|I_0^{m+k} \right|^{\frac{1}{2}}} \left| \left \langle f_1, \Phi_{P_{m+k}\left(r-p\right)} \right \rangle \right| \cdot  \frac{1}{\left|I_0^{m+2k} \right|^{\frac{1}{2}}} \left| \left \langle f_2, \Phi_{P_{m+2k}\left(2^k r+a+\frac{p^2}{2^m} \right)} \right \rangle \right| \nonumber \\
&& \quad \quad \quad \quad \quad \quad \cdot   \left| \left \langle f_3, \Phi_{P_{m+3k}(2^{2k}r+2^{k}a+b+\frac{p^3}{2^{2m}})} \right \rangle \right| \left | \left \langle f_4, \Phi_{P_{m+3k}(2^{2k}r+2^{k}a+b)} \right \rangle \right |\nonumber\\
&&\lesssim \frac{1}{2^m} \cdot \frac{1}{2^{m+3k}} \sum_{\substack{p \sim 2^m \\ r \in \Z}} \sum_{a, b \sim 2^k} \frac{\int_{I_{r-p}^{m+k}}|f_1|}{\left|I_0^{m+k}\right|} \cdot  \frac{\int_{I_{2^k r+a+\frac{p^2}{2^m}}^{m+2k}} |f_2|}{\left|I_0^{m+2k} \right|}  \cdot \frac{\int_{I^{m+3k}_{2^{2k}r+2^ka+b+\frac{p^3}{2^{2m}}}}|f_3|}{\left|I_0^{m+3k} \right|} \cdot \frac{\int_{I^{m+3k}_{2^{2k}r+2^ka+b}}|f_4|}{\left|I_0^{m+3k} \right|}\,. 
\end{eqnarray*}
Following the standard restricted weak-type estimate argument, we assume:
\begin{enumerate}
\item [$\bullet$] $f_j=\one_{F_j}$, where $F_j, j=1, 2, 3$ are any arbitrary finitely measurable sets;
\item [$\bullet$] $F_4$ is any finitely measurable set;
\item [$\bullet$] $\Omega:=\bigcup_{j=1}^3 \left\{M \one_{F_j} \ge \frac{C|F_j|}{|F_4|} \right\}$, where $C$ is sufficiently large such that $\left|F_4' \right| \ge \frac{|F_4|}{2}$ with $F_4':=F_4 \backslash \Omega$. 
\end{enumerate}

Letting $l=r-p$ and using the fact that
$$
\frac{\int_{I_{2^k r+a+\frac{p^2}{2^m}}^{m+2k}} |f_2|}{\left|I_0^{m+2k} \right|} \le 1, 
$$
we further have
\begin{equation} \label{20231009eq10}
\Lambda_m^{k} (\vec{f})\lesssim \frac{1}{2^{2m+3k}} \sum_{l \in \Z} \frac{\int_{I_l^{m+k}}|f_1|}{\left|I_0^{m+k} \right|} \left( \sum_{\substack{p \sim 2^m \\ a, b \sim 2^k}} \frac{\int_{I^{m+3k}_{2^{2k}l+2^{2k}p+2^ka+b+\frac{p^3}{2^{2m}}}} |f_3|}{\left|I_0^{m+3k} \right|} \cdot \frac{\int_{I^{m+3k}_{2^{2k}l+2^{2k}p+2^ka+b}} |f_4|}{\left|I_0^{m+3k} \right|} \right).
\end{equation} 
Next, we note that given $l \in \Z, p \sim 2^m, a, b \sim 2^k$ if $\int_{I^{m+3k}_{2^{2k}l+2^{2k}p+2^ka+b}} |f_4| \neq 0$ then $I^{m+3k}_{2^{2k}l+2^{2k}p+2^ka+b} \cap \Omega^c \neq \emptyset$ which further gives
$$
I_{\frac{l}{2^{m}}}^k \cap \Omega^c \neq \emptyset. 
$$
Using this and \eqref{20231009eq10}, we thus have 
\begin{equation} \label{20231009eq11}
\Lambda_m^{k}(\vec{f}) \lesssim  \frac{1}{2^{2m+3k}} \sum_{l \in \Z, \  I_{\frac{l}{2^{m}}}^k \cap \Omega^c \neq \emptyset } \frac{\int_{I_l^{m+k}}|f_1|}{\left|I_0^{m+k} \right|} \left( \sum_{\substack{p \sim 2^m \\ a, b \sim 2^k}} \frac{\int^{m+3k}_{I_{2^{2k}l+2^{2k}p+2^ka+b+\frac{p^3}{2^{2m}}}} |f_3|}{\left|I_0^{m+3k} \right|} \right).
\end{equation}

We consider two major cases.

\medskip 
\noindent\textsf{\underline{Case I}: $\max\{|F_1|, |F_3|\} \ge |F_4| $}

We split our analysis in two subcases
\medskip 

\noindent\textit{\underline{Case I.1}: if $|F_3| \ge |F_4|$}, then by \eqref{20231009eq11}, 
\begin{eqnarray*}
\Lambda_m^{k} (\vec{f})%
&\lesssim& \frac{1}{2^{2m+3k}} \cdot 2^{m+k} \cdot 2^{m+2k} \cdot \left(\sum_{l \in \Z} \int_{I_l^{m+k}} |f_1| \right) \\
& \lesssim& |F_1| \lesssim |F_1||F_3||F_4|^{-1}. 
\end{eqnarray*}

\medskip 

\noindent\textit{\underline{Case I.2}: if $|F_1| \ge |F_4|$}, then by \eqref{20231009eq11}, 
\begin{eqnarray*}
\Lambda_m^{k} (\vec{f})%
& \lesssim & \frac{1}{2^{2m+3k}} \sum_{l \in \Z} \sum_{p \sim 2^m} \sum_{a, b \sim 2^k} \frac{\int_{I^{m+3k}_{2^{2k}l+2^{2k}p+2^ka+b+\frac{p^3}{2^{2m}}}} |f_3|}{\left|I_0^{m+3k} \right|} \\
& \lesssim &  \frac{1}{2^m} \sum_{p \sim 2^m} \left( \sum_{l \in \Z} \sum_{a, b \sim 2^k} \int_{I^{m+3k}_{2^{2k}l+2^{2k}p+2^ka+b+\frac{p^3}{2^{2m}}}} |f_3| \right) \\
&\lesssim& |F_3| \lesssim |F_1||F_3||F_4|^{-1}.
\end{eqnarray*}

\medskip 

\noindent\textit{\underline{Case II}: $|F_1|, |F_3| <|F_4| $}

In this situation since $I_{\frac{l}{2^m}}^k \cap \Omega^c \neq \emptyset$ and $
\dist \left( I_{\frac{l}{2^m}}^k, I^{m+3k}_{2^{2k}l+2^{2k}p+2^ka+b+\frac{p^3}{2^{2m}}} \right) \leq \frac{2}{2^{k}}$, we have
$$
\left(5 \bigcup_{p' \sim 2^m, a', b' \sim 2^k} I^{m+3k}_{2^{2k}l+2^{2k}p'+2^ka'+b'+\frac{(p')^3}{2^{2m}}} \right) \cap \Omega^c \neq \emptyset.
$$
This implies
\begin{eqnarray*}
\Lambda_m^{k}(\vec{f})%
& \lesssim & \frac{1}{2^{2m+3k}} \cdot 2^{m+k} \cdot 2^{m+3k} \cdot \frac{1}{2^k} \sum_{\substack{l \in \Z \\ I_{\frac{l}{2^m}}^k \cap \Omega^c \neq \emptyset}} \int_{I_l^{m+k}} |f_1| \left( \frac{1}{\left|I_{\frac{l}{2^m}}^k\right|} \sum_{\substack{p \sim 2^m \\ a, b \sim 2^k}} \int_{I^{m+3k}_{2^{2k}l+2^{2k}p+2^ka+b+\frac{p^3}{2^{2m}}}}|f_3| \right) \\
&\lesssim& \frac{|F_3|}{|F_4|} \cdot \sum_{l \in \Z} \int_{I_l^{m+k}} |f_1| \lesssim |F_1||F_3||F_4|^{-1}. 
\end{eqnarray*}

\medskip

Therefore, we have shown that
$$
\Lambda_m^{k}(\vec{f}) \lesssim |F_1||F_3||F_4|^{-1}.
$$
Similar reasonings also give
$$
\Lambda_m^{k}(\vec{f}) \lesssim |F_1||F_2||F_4|^{-1} \quad \textrm{and} \quad\Lambda_m^{k}(\vec{f}) \lesssim |F_2||F_3||F_4|^{-1}\,.
$$
Once at this point, via some standard interpolation arguments we obtain the desired conclusion.

\end{proof}

\subsection{Treatment of the off-diagonal terms for $k \ge 0$} \label{secoffdiag}
Recall that the off-diagonal cases refer to the situation when $\max\{j, l, m\}-\min\{j, l, m\}>300$, and there are two different sub-cases under this situation. 

\subsubsection{Treatment of the stationary off-diagonal cases} \label{20240219subsubsec01}
 In what follows, we only consider the case when $j=l>m+100$, while the proof of the other stationary off-diagonal cases follow from simple adaptations of the present case. To this end, we recall that the single-scale operator in the case under consideration is
$$
\calT_{j, m}^k(f_1, f_2, f_3)(x):=\int_{\R}  \left(f_1* \check{\phi}_{j+k}\right)(x-t) \left(f_2* \check{\phi}_{j+2k}\right)(x+t^2) \left(f_3*\check{\phi}_{m+3k}\right)(x+t^3) \rho_k(t)dt.
$$

We consider three situations.

\medskip 

\noindent \fbox{\textsf{Case I: Treatment for the case $0 \le k \le \frac{j-m}{2}$.}} Note that in this case, the dominant frequency scale has the magnitude $\sim 2^{j+2k}$, and therefore
\begin{eqnarray} \label{20231020eq01}
\left|\Lambda_{j, m}^k(\vec{f})\right|&:=&\left|\int_{\R^2}   \left(f_1* \check{\phi}_{j+k}\right)(x-t) \left(f_2* \check{\phi}_{j+2k}\right)(x+t^2) \left(f_3*\check{\phi}_{m+3k}\right)(x+t^3) \left(f_4 * \check{\phi}_{j+2k} \right)(x)\rho_k(t)dtdx\right|\nonumber \\
&\lesssim& \frac{1}{2^{2j+2k}} \sum_{\substack{z \in \Z \\ p \sim 2^j}} \left| \left(f_1*\check{\phi}_{j+k} \right) \left(\frac{z}{2^{j+2k}}-\frac{p}{2^{j+k}} \right) \right| \left| \left(f_2 * \check{\phi}_{j+2k} \right) \left(\frac{z}{2^{j+2k}}+\frac{p^2}{2^{2j+2k}} \right) \right| \nonumber  \\
&& \quad \quad \quad \quad \quad  \cdot \left| \left(f_3*\check{\phi}_{m+3k} \right) \left( \frac{z}{2^{j+2k}}+\frac{p^3}{2^{3j+3k}} \right) \right| \left| \left( f_4 * \check{\phi}_{j+2k} \right) \left(\frac{z}{2^{j+2k}} \right) \right|. 
\end{eqnarray}
Since $0 \le k \le \frac{j-m}{2}$, we have $2^{-j-2k} \le 2^{-j-k} \le 2^{-m-3k}$ and hence \eqref{20231020eq01} may be rewritten as
\begin{eqnarray} \label{20231021eq01}
&& \left|\Lambda_{j, m}^k(\vec{f})\right|\lesssim \frac{1}{2^{\frac{j-m}{2}}} \cdot 2^{2k} \sum_{\substack{r \in \Z \\ p \sim 2^j}} \left| \left \langle f_3, \Phi_{P_{m+3k} \left(r+\frac{p^3}{2^{3j-m}} \right)} \right \rangle \right| \cdot \Bigg[ \sum_{I_u^{j+k} \subseteq I_r^{m+3k}} \left| \left \langle f_1, \Phi_{P_{j+k} \left(u-p \right)} \right \rangle \right | \nonumber \\
&& \quad \quad \quad \quad \quad \quad \cdot \left( \sum_{I_v^{j+2k} \subseteq I_u^{j+k}} \left| \left \langle f_2, \Phi_{P_{j+2k} \left(v+\frac{p^2}{2^j} \right)} \right \rangle \right | \left| \left \langle f_4, \Phi_{P_{j+2k}(v)} \right \rangle \right| \right)  \Bigg ] \nonumber \\
&& =\frac{1}{2^{\frac{j-m}{2}}} \cdot 2^{2k} \sum_{\substack{ r \in \Z \\ p \sim 2^j}} \sum_{\substack{ a \sim 2^{j-m-2k} \\ b \sim 2^k}} \left| \left \langle f_3, \Phi_{P_{m+3k} \left(r+\frac{\frac{p^3}{2^{2j}}}{2^{j-m}} \right)} \right \rangle \right| \left| \left \langle f_1, \Phi_{P_{j+k} \left(2^{j-m-2k}r+a-p \right)} \right \rangle \right| \nonumber \\
&& \quad \quad \quad \quad \quad \quad \cdot \left| \left \langle f_2, \Phi_{P_{j+2k} \left(2^{j-m-k}r+2^ka+b+\frac{p^2}{2^j} \right)} \right \rangle \right|  \left| \left \langle f_4, \Phi_{P_{j+2k} \left(2^{j-m-k}r+2^ka+b \right)} \right \rangle \right|\,.
\end{eqnarray}

As usual, we let
\begin{enumerate}
\item [$\bullet$] $f_s=\one_{F_s}, s=1, 2, 3$ for $F_s$ being any arbitrary finitely measurable sets; 
\item [$\bullet$] $F_4$ be any arbitrary finitely measurable set and $\Omega:=\bigcup\limits_{s=1}^3 \left\{M\one_{F_s} \ge \frac{C|F_s|}{|F_4|} \right\}$ for $C$ being sufficiently large such that $|F_4'| \ge \frac{|F_4|}{2}$, where $F_4':=F_4 \backslash \Omega$. 
\end{enumerate}

Then by \eqref{20231021eq01}, we have
\begin{equation}\label{bounda}
\left|\Lambda_{j, m}^k(\vec{f}) \right|\lesssim \frac{1}{2^{2j+2k}} \sum_{\substack{r \in \Z \\ p \sim 2^j}} \sum_{\substack{a \sim 2^{j-m-2k} \\ b \sim 2^{k}}} \frac{\int_{I_{r+\frac{\frac{p^3}{2^{2j}}}{2^{j-m}}}^{m+3k}} |f_3|}{\left|I_0^{m+3k} \right|} \cdot \frac{\int_{I_{2^{j-m-2k}r+a-p}^{j+k}} |f_1|}{\left|I_0^{j+k} \right|} \cdot \frac{\int_{I_{2^{j-m-k}r+2^ka+b+\frac{p^2}{2^j}}^{j+2k}}|f_2|}{\left|I_0^{j+2k} \right|} \cdot \frac{\int_{I_{2^{j-m-k}r+2^ka+b}^{j+2k}}|f_4|}{\left|I_0^{j+2k} \right|}.
\end{equation}
Note that we trivially have
\begin{equation} \label{20231021eq02}
\frac{\int_{I_{2^{j-m-k}r+2^ka+b+\frac{p^2}{2^j}}^{j+2k}}|f_2|}{\left|I_0^{j+2k} \right|}  \lesssim 1, 
\end{equation}
and hence, after applying the change of variable $l=r+\frac{\frac{p^3}{2^{2j}}}{2^{j-m}}$, we deduce 
\begin{equation} \label{20231021eq03}
\left|\Lambda_{j, m}^k(\vec{f})\right| \lesssim \frac{1}{2^{2j+2k}} \sum_{l \in \Z}\frac{\int_{I_l^{m+3k}}|f_3|}{\left|I_0^{m+3k} \right|}\cdot  \sum_{\substack{p \sim 2^j \\ a \sim 2^{j-m-2k} \\ b \sim 2^{k}}}  \frac{\int_{I_{2^{j-m-2k}\left(l-\frac{\frac{p^3}{2^{2j}}}{2^{j-m}} \right) +a-p}^{j+k}} |f_1|}{\left|I_0^{j+k} \right|} \cdot \frac{\int_{I_{2^{j-m-k}\left(l-\frac{\frac{p^3}{2^{2j}}}{2^{j-m}} \right) +2^ka+b}^{j+2k}}|f_4|}{\left|I_0^{j+2k} \right|}.
\end{equation}
Finally, since $\int_{I_{2^{j-m-k}\left(l-\frac{\frac{p^3}{2^{2j}}}{2^{j-m}} \right) +2^ka+b}^{j+2k}}|f_4| \neq 0$, we have
$I_{2^{j-m-k}\left(l-\frac{\frac{p^3}{2^{2j}}}{2^{j-m}} \right) +2^ka+b}^{j+2k} \cap \Omega^c \neq \emptyset$, which further gives $I_{\frac{l}{2^{m+2k}}}^k \cap \Omega^c \neq \emptyset$; thus, by \eqref{20231021eq03}, we have
\begin{equation} \label{20201021eq04}
\left|\Lambda_{j, m}^k (\vec{f})\right| \lesssim \frac{1}{2^{2j+k}} \sum_{\substack{l \in \Z \\ I_{\frac{l}{2^{m+2k}}}^k \cap \Omega^c \neq \emptyset}}\frac{\int_{I_l^{m+3k}}|f_3|}{\left|I_0^{m+3k} \right|}\cdot  \sum_{\substack{p \sim 2^j \\ a \sim 2^{j-m-2k} }}  \frac{\int_{I_{2^{j-m-2k}\left(l-\frac{\frac{p^3}{2^{2j}}}{2^{j-m}} \right) +a-p}^{j+k}} |f_1|}{\left|I_0^{j+k} \right|}.
\end{equation}  
From here on, one applies the same reasonings as those in Section \ref{20231112sec01} in order to deduce that
\begin{equation}\label{bdhalf}
\left|\Lambda_{j, m}^{k}(\vec{f})\right| \lesssim |F_1|^{\theta_1}|F_2|^{\theta_2} |F_3|^{\theta_3} |F_4|^{1-\theta_1-\theta_2-\theta_3}, 
\end{equation}
where $0<\theta_1, \theta_2, \theta_3<1$ with $\theta_1+\theta_2+\theta_3\leq 2$.

\vspace{0.1cm}

\noindent\fbox{{\textsf{Case II: Treatment for the case $\frac{j-m}{2}<k \le j-m$.}}}  The treatment of this case is similar to the case when $0 \le k \le \frac{j-m}{2}$. Indeed, based on  $\frac{j-m}{2}< k \le j-m$, we have $2^{-j-2k} \le 2^{-m-3k} \le 2^{-j-k}$ and hence the dominant frequency scale remains $2^{j+2k}$ while the intermediate regime is now flipped. Consequently, in the same spirit with the decomposition provided at Case I, one deduces that (compare with \eqref{20231021eq01} and \eqref{bounda})
\begin{eqnarray}\label{case2diagofstat}
&&\left|\Lambda_{j, m}^k(\vec{f})\right|\lesssim\frac{1}{2^{\frac{j-m}{2}}} \cdot 2^{2k} \sum_{\substack{r \in \Z \\ p\sim 2^j}} \left| \left \langle f_1, \Phi_{P_{j+k} \left(r-p \right)} \right \rangle \right| \cdot \Bigg[ \sum_{I_u^{m+3k} \subseteq I_r^{j+k}} \left| \left \langle f_3, \Phi_{P_{m+3k} \left(u+\frac{p^3}{2^{3j-m}} \right)} \right \rangle \right| \nonumber\\
&&\quad \quad \quad \quad \quad \quad \cdot \left( \sum_{I_v^{j+2k} \subseteq I_u^{m+3k}}  \left| \left \langle f_2, \Phi_{P_{j+2k} 
\left(v+\frac{p^2}{2^j} \right)} \right \rangle \right| \left| \left \langle f_4, \Phi_{P_{j+2k} \left(v \right)} \right \rangle \right|  \right) \\
&&\lesssim \frac{1}{2^{2j+2k}} \sum_{\substack{r \in \Z \\ p \sim 2^j}} \sum_{\substack{a \sim 2^{2k+m-j} \\ b \sim 2^{j-m-k}}} \frac{\int_{I_{r-p}^{j+k}}|f_1|}{\left|I_0^{j+k} \right|}  \cdot \frac{\int_{I^{m+3k}_{2^{2k+m-j}r+a+\frac{\frac{p^3}{2^{2j}}}{2^{j-m}}}}|f_3|}{\left|I_0^{m+3k} \right|} \cdot \frac{\int_{I_{2^k r+2^{j-m-k}a+b+\frac{p^2}{2^j}}^{j+2k}} |f_2|}{\left|I_0^{j+2k} \right|}  \cdot \frac{\int_{I_{2^k r+2^{j-m-k}a+b}^{j+2k}} |f_4|}{\left|I_0^{j+2k} \right|}.\nonumber
\end{eqnarray}
The rest of reasonings are now mirroring the ones provided at Case I---we leave further details to the interested reader---which implies the same conclusion as in \eqref{bdhalf}.

\medskip

\noindent\fbox{\textsf{Case III: Treatment for the case $j-m \le k \le j$.}} Since $k \ge j-m$, we have $2^{-m-3k} \le 2^{-j-2k} \le 2^{-j-k}$ which implies that the dominant frequency scale is now $\sim 2^{m+3k}$. Therefore, in parallel to \eqref{case2diagofstat}, we deduce
\begin{eqnarray*}
&& \left|\Lambda_{j, m}^k (\vec{f})\right|\lesssim \frac{1}{2^j} \sum_{\substack{z \in \Z \\ p \sim 2^j}}  \frac{1}{\left|I_0^{j+k} \right|^{\frac{1}{2}}} \left| \left\langle f_1, \Phi_{P_{j+k} \left(r-p\right)} \right \rangle \right| \cdot \bigg[ \sum_{I_u^{j+2k} \subseteq I_r^{j+k}} \frac{1}{\left|I_0^{j+2k} \right|^{\frac{1}{2}}} \left| \left\langle f_2, \Phi_{P_{j+2k} \left(u+\frac{p^2}{2^j} \right)} \right \rangle \right | \\
&& \quad \quad \quad \quad \quad \quad  \cdot \left( \sum_{I_v^{m+3k} \subseteq I_u^{j+2k}} \left| \left \langle f_3, \Phi_{P_{m+3k} \left(v+\frac{p^3}{2^{3j-m}} \right)} \right \rangle \right| \left| \left \langle f_4, \Phi_{P_{m+3k}(v)} \right \rangle \right| \right) \bigg] \\
&& \lesssim \frac{1}{2^{j+m+3k}} \sum_{\substack{r \in \Z \\ p \sim 2^j}} \sum_{\substack{a \sim 2^k \\ b \sim 2^{k+m-j}}} \frac{\int_{I_{r-p}^{j+k}}|f_1|}{\left|I_0^{j+k} \right|} \cdot \frac{\int_{I_{2^k r+a+\frac{p^2}{2^j}}} |f_2|}{\left|I_0^{j+2k} \right|} \cdot \frac{\int_{I_{2^{2k+m-j}r+2^{k+m-j}a+b+\frac{\frac{p^3}{2^{2j}}}{2^{j-m}}}}|f_3|}{\left|I_0^{m+3k} \right|}  \cdot \frac{\int_{I_{2^{2k+m-j}r+2^{k+m-j}a+b}}|f_4|}{\left|I_0^{m+3k} \right|}.
\end{eqnarray*}
Again, applying similar reasonings with the ones in Case I we obtain that \eqref{bdhalf} holds.

\medskip 

\noindent\fbox{\textsf{Case IV: Treatment for the case $k \ge j$.}} In this case, since the dominant frequency scale is $\sim 2^{m+3k}$, we let
$$ 
\left|\Lambda_{j, m}^{\ge}(\vec{f})\right|:=\left|\sum_{k \ge j}  \int_{\R^2}  \left(f_1* \check{\phi}_{j+k}\right)\left(x-\frac{t}{2^k} \right) \left(f_2* \check{\phi}_{j+2k}\right)\left(x+\frac{t^2}{2^{2k}} \right) \left(f_3*\check{\phi}_{m+3k}\right)\left(x+\frac{t^3}{2^{3k}} \right) \left(f_4 * \check{\phi}_{m+3k} \right)(x)dtdx\right|\,.
$$
Following the approach in Section \ref{MaindiaglargeK}---compare \eqref{bddiagklarge} with the expression below---we have
\begin{eqnarray*}
&& \left|\Lambda_{j, m}^{\ge}(\vec{f})\right| \lesssim \sum_{\substack{k \ge j \\ r \in \Z}} \frac{1}{2^j} \sum_{s \sim 2^j} \frac{1}{\left|I_0^{j+k} \right|^{\frac{1}{2}}} \left| \left \langle f_1, \Phi_{P_{j+k}(r-p)} \right \rangle \right|  \cdot \Bigg [\sum_{I_u^{j+2k} \subseteq I_r^{j+k}} \frac{1}{\left|I_0^{j+2k} \right|^{\frac{1}{2}}} \left| \left \langle f_2, \Phi_{P_{j+2k}\left(u+\frac{p^2}{2^j} \right)} \right \rangle \right| \nonumber \\
&&  \quad \quad \quad \quad \quad \quad \quad \quad \cdot \left( \sum_{I_v^{m+3k} \subseteq I_u^{j+2k}} \left| \left \langle f_3, \Phi_{P_{m+3k} \left(v+\frac{\frac{p^3}{2^{2j}}}{2^{j-m}} \right)}  \right \rangle \right| \left| \left \langle f_4, \Phi_{P_{m+3k}(v)} \right \rangle \right| \right) \Bigg ].
\end{eqnarray*}
Observe that if $I_v^{m+3k} \subseteq I_u^{j+2k}$, then by the assumption $k \ge j\geq m$, one has $I^{m+3k}_{v+\frac{\frac{p^3}{2^{2j}}}{2^{j-m}}} \subseteq 2I_u^{j+2k}$. Thus
\begin{eqnarray*}
&& \left|\Lambda_{j, m}^{\ge}(\vec{f})\right| \lesssim \sum_{\substack{k \ge j \\ r \in \Z}} \frac{1}{2^j} \sum_{s \sim 2^j} \frac{1}{\left|I_0^{j+k} \right|^{\frac{1}{2}}} \left| \left \langle f_1, \Phi_{P_{j+k}(r-p)} \right \rangle \right|  \cdot \Bigg [\sum_{I_u^{j+2k} \subseteq I_r^{j+k}} \frac{1}{\left|I_0^{j+2k} \right|^{\frac{1}{2}}} \left| \left \langle f_2, \Phi_{P_{j+2k}\left(u+\frac{p^2}{2^j} \right)} \right \rangle \right| \nonumber \\
&&  \quad \quad \quad \quad \quad \quad \quad \quad \cdot \left( \sum_{I_v^{m+3k} \subseteq 3 I_u^{j+2k}} \left| \left \langle f_3, \Phi_{P_{m+3k} \left(v \right)}  \right \rangle \right| \left| \left \langle f_4, \Phi_{P_{m+3k}(v-\frac{\frac{p^3}{2^{2j}}}{2^{j-m}})} \right \rangle \right| \right) \Bigg ] \\
&& \quad \approx\int_{\R} \sum_{\substack{k \ge j \\ r \in \Z}} \sum_{\substack{I_u^{j+2k} \subseteq I_r^{j+k} \\ I_v^{m+3k} \subseteq I_u^{j+2k}}} \left( \frac{1}{2^j} \sum_{s \sim 2^j} \frac{\left| \left \langle f_1, \Phi_{P_{j+k} \left(r-p \right)} \right \rangle \right|}{\left|I_r^{j+k} \right|^{\frac{1}{2}}} \frac{\left| \left \langle f_2, \Phi_{P_{j+2k} \left(u+\frac{p^2}{2^j} \right)} \right \rangle \right|}{\left|I_u^{j+2k} \right|^{\frac{1}{2}}} \frac{\left| \left \langle f_4, \Phi_{P_{m+3k} \left(v-\frac{\frac{p^3}{2^{2j}}}{2^{j-m}} \right)} \right \rangle \right|}{\left|I_v^{m+3k} \right|^{\frac{1}{2}}} \right) \\
&& \quad \quad \quad \cdot \frac{\left| \left \langle f_3, \Phi_{P_{m+3k}(v)} \right \rangle \right|}{\left|I_v^{m+3k} \right|^{\frac{1}{2}}} \one_{I_v^{m+3k}}(x)dx.
\end{eqnarray*}
By Cauchy-Schwarz, we see that
$$
\left|\Lambda_{j, m}^{\ge}(\vec{f})\right| \lesssim \int_{\R} \calT_{j, m}(f_1, f_2, f_4)(x) \calS f_3(x) dx,
$$
where
\begin{eqnarray*}
&& \calT_{j, m}^2(f_1, f_2, f_4)(x)\\
&& :=\sum_{\substack{I_u^{j+2k} \subseteq I_r^{j+k} \\ I_v^{m+3k} \subseteq I_u^{j+2k}}} \left( \frac{1}{2^j} \sum_{s \sim 2^j} \frac{\left| \left \langle f_1, \Phi_{P_{j+k} \left(r-p \right)} \right \rangle \right|}{\left|I_r^{j+k} \right|^{\frac{1}{2}}} \frac{\left| \left \langle f_2, \Phi_{P_{j+2k} \left(u+\frac{p^2}{2^j} \right)} \right \rangle \right|}{\left|I_u^{j+2k} \right|^{\frac{1}{2}}} \frac{\left| \left \langle f_4, \Phi_{P_{m+3k} \left(v-\frac{\frac{p^3}{2^{2j}}}{2^{j-m}} \right)} \right \rangle \right|}{\left|I_v^{m+3k} \right|^{\frac{1}{2}}} \right)^2\one_{I_v^{m+3k}}(x).
\end{eqnarray*}
Following now a similar argument with the one employed for the proof of Proposition \ref{20231227prop01}, we have
\begin{prop} 
Let $j, m \in \N$ with $j \ge m+100$ and $F_1, F_2, F_4$ be measurable sets with finite (nonzero) Lebesgue measure. Then 
\begin{center}
$\exists~F_4' \subseteq F_4$ measurable with $|F_4'| \ge \frac{1}{2} |F_4|$, 
\end{center}
such that for any triple of functions $f_1, f_2$ and $f_4$ obeying
$$
|f_1| \le \one_{F_1}, \quad |f_2| \le \one_{F_2}, \quad |f_4| \le \one_{F_4'}, 
$$
and any $1<p<\infty$, one has that
$$
\left\|\calT_{j, m}(f_1, f_2, f_4) \right\|_{L^p(\R)} \lesssim_p j^{10} \frac{|F_1|}{|F_4|} \frac{|F_2|}{|F_4|} |F_4|^{\frac{1}{p}}.
$$
\end{prop}
Finally, we conclude that
$$
\left|\Lambda_{j, m}^{\ge} \left(\vec{f} \right)\right| \lesssim \left\|\calT(f_1, f_2, f_4) \right\|_{L^p(\R)} \left\|\calS f_3 \right\|_{L^{p'}(\R)} \lesssim_p j^{10} \frac{|F_1|}{|F_4} \frac{|F_2|}{|F_4|} |F_4|^{\frac{1}{p}} |F_3|^{1-\frac{1}{p}}, \quad \forall p \in (1, \infty). 
$$

\subsubsection{Treatment of the non-stationary off-diagonal case}\label{nonsthi}
Recall by \cite[Lemma 2.3]{HL23} that
\begin{eqnarray} \label{20231030eq01}
\left| \Lambda_{j, l, m}^k(\vec{f}) \right|%
&\lesssim& \frac{1}{2^n} \int_{\R \times \left\{\frac{1}{4} \le |t| \le 4 \right\}} \left| \left(\widehat{f_1} \left(\cdot \right) \phi \left( \frac{\cdot}{2^{j+k}} \right) \right)^{\vee}  \left(x-\frac{t}{2^k} \right)\right|\left| \left(\widehat{f_2} \left(\cdot \right) \phi \left( \frac{\cdot}{2^{l+2k}} \right) \right)^{\vee}  \left(x+\frac{t^2}{2^{2k}} \right)\right| \nonumber \\
&& \quad \cdot \left| \left(\widehat{f_3} \left(\cdot \right) \phi \left( \frac{\cdot}{2^{m+3k}} \right) \right)^{\vee}  \left(x+\frac{t^3}{2^{3k}} \right)\right| \left| \left(\widehat{f_4} \left(\cdot \right) \one_{A(k, j, l, m)} \left( \cdot \right) \right)^{\vee}  \left(x \right)\right|,
\end{eqnarray}
where $n:=\max\{j, l, m\}$ and $A(k, j, l , m):=\supp \ \phi \left( \frac{\cdot}{2^{j+k}} \right)+\supp \ \phi \left(\frac{\cdot}{2^{l+2k}} \right)+\supp \ \phi \left(\frac{\cdot}{2^{m+3k}} \right)$. 

Let $f_1, f_2, f_3$ and $f_4$ be defined as usual. Following the same spirit as in the treatment of the stationary diagonal term, see Sections \ref{MaindiaglargeK} and \ref{20231112sec01}, we deduce that for any $j, l, m \ge 0$ and $\vec{\theta}=(\theta_1, \theta_2, \theta_3)$ with $0<\theta_1, \theta_2, \theta_3<1$ with $\theta_1+\theta_2+\theta_3\leq 2$ we have
$$
\left| \sum_{k \ge 0}  \Lambda_{j, l, m}^k (\vec{f}) \right| \lesssim_{\vec{\theta}} \frac{n^3}{2^n} \cdot |F_1|^{\theta_1}|F_2|^{\theta_2} |F_3|^{\theta_3} |F_4|^{1-\theta_1-\theta_2-\theta_3}$$
which finally implies
\begin{eqnarray*}
\left| \Lambda_{k \ge 0}^{\not D, NS} (\vec{f}) \right|%
&:=& \left| \sum_{\substack{ j, l, m \ge 0 \\ \textrm{non-stationary off-diagonal}}} \sum_{k \ge 0}  \Lambda_{j, l, m}^k (\vec{f}) \right|\lesssim_{\vec{\theta}}   \left(\sum_{j, l, m \ge 0} \frac{n^3}{2^n}  \right) \cdot |F_1|^{\theta_1}|F_2|^{\theta_2} |F_3|^{\theta_3} |F_4|^{1-\theta_1-\theta_2-\theta_3}\\
&\lesssim&  |F_1|^{\theta_1}|F_2|^{\theta_2} |F_3|^{\theta_3} |F_4|^{1-\theta_1-\theta_2-\theta_3}\,..  
\end{eqnarray*}

\subsection{A brief overview of the high frequency component in the regime $k \le 0$.} \label{kneg}

In this section, we present a very brief outline of the regime $k$ negative. We will focus only on the main diagonal part as the modifications required for treating the off-diagonal case are relatively simpler (and following Section \ref{secoffdiag}).   

To begin with, we notice that the dominant frequency scale in this regime is $2^{m+k}$. Therefore
 $$
 \Lambda_m^k (\vec{f})=\int_{\R^2} \left(f_1* \check{\phi}_{m+k}\right)\left(x-t \right) \left(f_2* \check{\phi}_{m+2k}\right)\left(x+t^2 \right) \left(f_3*\check{\phi}_{m+3k}\right)\left(x+t^3 \right) \left(f_4*\check{\phi}_{m+k}\right)(x) \rho(t)dtdx, 
 $$
which, by applying the change variables $t^3 \to t$ first and then $2^{3k} t \to t$, becomes
$$
\int_{\R^2} \left(f_1*\check{\phi}_{m+k} \right) \left(x-\frac{t^{\frac{1}{3}}}{2^k} \right) \left(f_2*\check{\phi}_{m+2k} \right) \left(x+\frac{t^{\frac{2}{3}}}{2^{2k}} \right)\left(f_3*\check{\phi}_{m+3k} \right) \left(x+\frac{t}{2^{3k}} \right) \left(f_4*\check{\phi}_{m+k} \right)(x) \rho(t)dtdx. 
$$
Now we apply the discretization in $x$ first, which gives
\begin{eqnarray} \label{20231111eq01}
\Lambda_m^k(\vec{f})%
& \simeq& \sum_{z\in\Z} \frac{1}{2^{m+k}} \int_{\R} \left(f_1*\check{\phi}_{m+k} \right)\left(\frac{z}{2^{m+k}}-\frac{t^{\frac{1}{3}}}{2^k} \right) \left(f_2*\check{\phi}_{m+2k} \right) \left(\frac{z}{2^{m+k}}+\frac{t^{\frac{2}{3}}}{2^{2k}} \right) \nonumber \\
&& \quad \quad \left(f_3*\check{\phi}_{m+3k} \right) \left(\frac{z}{2^{m+k}}+\frac{t}{2^{3k}} \right) \left(f_4*\check{\phi}_{m+k} \right)\left(\frac{z}{2^{m+k}} \right) \rho(t)dt.
\end{eqnarray}
Again, we consider two different cases.

\subsubsection{The case $-k \ge m$}

We apply a discretization in $t$ to see that 
\begin{eqnarray} \label{20231112eq02}
&& \Lambda_m^{\le}(\vec{f}):= \frac{1}{2^m} \sum_{s \sim 2^m} \sum_{\substack{-k \ge m \\ z \in \Z}} \frac{1}{2^{k+m}} \left| \left(f_3*\check{\phi}_{m+3k} \right) \left(\frac{z}{2^{m+k}}+\frac{s}{2^{m+3k}} \right) \right| \left| \left(f_4*\check{\phi}_{m+k} \right) \left(\frac{z}{2^{m+3k}} \right) \right| \nonumber \\
&& \quad \quad \quad \quad \quad \quad \cdot \left| \left(f_1*\check{\phi}_{m+k} \right) \left(\frac{z}{2^{m+k}}-\frac{s^{\frac{1}{3}}}{2^{\frac{m}{3}+k}} \right) \right| \left| \left(f_2* \check{\phi}_{m+2k} \right) \left(\frac{z}{2^{m+k}}+\frac{s^{\frac{2}{3}}}{2^{\frac{2m}{3}+2k}} \right) \right| \nonumber \\
&&\quad = \sum_{\substack{-k \ge m \\ r \in \Z}} \frac{1}{2^m} \sum_{s \sim 2^m} \frac{1}{\left|I_r^{m+3k} \right|^{\frac{1}{2}}} \left| \left\langle f_3, \Phi_{P_{m+3k} \left(r+s \right)} \right \rangle \right| \Bigg [ \sum_{I_u^{m+2k} \subseteq I_r^{m+3k}} \frac{1}{\left|I_u^{m+2k} \right|^{\frac{1}{2}}} \left| \left \langle f_2, \Phi_{P_{m+2k} \left(u+2^{\frac{m}{3}} s^{\frac{2}{3}} \right)} \right \rangle \right| \nonumber   \\
&& \quad \quad \quad \quad \cdot \left( \sum_{I_v^{m+k} \subseteq I_u^{m+2k}} \left| \left \langle f_1, \Phi_{P_{m+k} \left(v-2^{\frac{2m}{3}}s^{\frac{1}{3}} \right)} \right \rangle \right| \left| \left \langle f_4, \Phi_{P_{m+k}(v)} \right \rangle \right| \right) \Bigg ]. 
\end{eqnarray}
Note that since $I_v^{m+k} \subseteq I_u^{m+2k}$, using the assumption that $-k \ge m$, we see that $I^{m+k}_{v-2^{\frac{2m}{3}}s^{\frac{1}{3}}} \subseteq 2I_u^{m+2k}$ and hence
\begin{eqnarray*}
&& \Lambda_m^{\le}(\vec{f})\approx\sum_{\substack{-k \ge m \\ r \in \Z}} \frac{1}{2^m} \sum_{s \sim 2^m} \frac{1}{\left|I_r^{m+3k} \right|^{\frac{1}{2}}} \left| \left\langle f_3, \Phi_{P_{m+3k} \left(r+s \right)} \right \rangle \right| \Bigg [ \sum_{I_u^{m+2k} \subseteq I_r^{m+3k}} \frac{1}{\left|I_u^{m+2k} \right|^{\frac{1}{2}}} \left| \left \langle f_2, \Phi_{P_{m+2k} \left(u+2^{\frac{m}{3}} s^{\frac{2}{3}} \right)} \right \rangle \right| \nonumber   \\
&& \quad \quad \quad \quad \cdot \left( \sum_{I_v^{m+k} \subseteq I_u^{m+2k}} \left| \left \langle f_1, \Phi_{P_{m+k} \left(v \right)} \right \rangle \right| \left| \left \langle f_4, \Phi_{P_{m+k}\left(v+2^{\frac{2m}{3}}s^{\frac{1}{3}}\right)} \right \rangle \right| \right) \Bigg ] \\
&& \quad =\int_{\R} \sum_{\substack{-k \ge m \\ r \in \Z}} \sum_{\substack{I_u^{m+2k} \subseteq I_r^{m+3k} \\ I_v^{m+k} \subseteq I_u^{m+2k}}} \left( \frac{1}{2^m} \sum_{s \sim 2^m} \frac{\left |\left \langle f_3, \Phi_{P_{m+3k} \left(r+s \right)} \right \rangle \right|}{\left|I_r^{m+3k} \right|^{\frac{1}{2}}}  \frac{ \left| \left \langle f_2, \Phi_{P_{m+2k} \left(u+2^{\frac{m}{3}} s^{\frac{2}{3}} \right)} \right \rangle \right|}{\left|I_u^{m+2k} \right|^{\frac{1}{2}}} \frac{\left| \left \langle f_4, \Phi_{P_{m+k}\left(v+2^{\frac{2m}{3}}s^{\frac{1}{3}}\right)} \right \rangle \right| }{\left|I_v^{m+k} \right|^{\frac{1}{2}}} \right) \\
&& \quad \quad \quad \quad \one_{I_v^{m+k}}(x) \cdot \frac{\left| \left \langle f_1, \Phi_{P_{m+k}(v)} \right \rangle \right|}{\left|I_v^{m+k} \right|^{\frac{1}{2}}} \one_{I_v^{m+k}}(x) dx \lesssim \int_{\R} \calT(f_3, f_2, f_4)(x) \calS f_1(x)dx, 
\end{eqnarray*}
where
$$\calT^2(f_3, f_2, f_4)(x):=$$ 
$$\sum_{\substack{-k \ge m \\ r \in \Z}} \sum_{\substack{I_u^{m+2k} \subseteq I_r^{m+3k} \\ I_v^{m+k} \subseteq I_u^{m+2k}}} \left( \frac{1}{2^m} \sum_{s \sim 2^m} \frac{\left |\left \langle f_3, \Phi_{P_{m+3k} \left(r+s \right)} \right \rangle \right|}{\left|I_r^{m+3k} \right|^{\frac{1}{2}}}  \frac{ \left| \left \langle f_2, \Phi_{P_{m+2k} \left(u+2^{\frac{m}{3}} s^{\frac{2}{3}} \right)} \right \rangle \right|}{\left|I_u^{m+2k} \right|^{\frac{1}{2}}} \frac{\left| \left \langle f_4, \Phi_{P_{m+k}\left(v+2^{\frac{2m}{3}}s^{\frac{1}{3}}\right)} \right \rangle \right| }{\left|I_v^{m+k} \right|^{\frac{1}{2}}} \right)^2 \one_{I_v^{m+k}}(x).$$ 

Following now the argument presented in Proposition \ref{20231227prop01}, one can show that for any $F_2, F_3, F_4$ measurable sets with finite (nonzero) Lebesgue measure, there exists some $F_4' \subseteq F_4$ measurable with $|F_4'| \ge \frac{1}{2}|F_4|$, such that for any triple of measurable functions $\{f_j\}_{j=2}^{4}$ obeying $|f_2| \le \one_{F_2}, |f_3| \le \one_{F_3}$ and $f_4 \le \one_{F_4'}$ and any $1<p<\infty$, one has 
$$
\left\|\calT(f_3, f_2, f_4) \right\|_{L^p(\R)} \lesssim_p m^{10} \frac{|F_3|}{|F_4|} \frac{|F_2|}{|F_4|} \left|F_4 \right|^{\frac{1}{p}}.
$$
Therefore, we have 
$$
\Lambda_m^{\le}(\vec{f}) \lesssim_p m^{10} \frac{|F_3|}{|F_4|}\frac{|F_2|}{|F_4|} \left|F_4 \right|^{\frac{1}{p}} |F_1|^{1-\frac{1}{p}}, \quad \forall p \in (1, \infty).
$$

\subsubsection{The case $0 \le -k \le m$} After applying a further parametrization of \eqref{20231112eq02} (for a single $k$), we see that it suffices to consider the term 
\begin{eqnarray*}
&& \Lambda_m^{k}(\vec{f}) \simeq \frac{1}{2^m} \sum_{s \sim 2^m} \sum_{r \in \Z} \sum_{a, b \sim 2^{-k}} \frac{1}{\left|I_r^{m+3k} \right|^{\frac{1}{2}}} \left| \left \langle f_3, \Phi_{P_{m+3k}(r+s) } \right \rangle \right| \cdot \frac{1}{\left| I_0^{m+2k} \right|^{\frac{1}{2}}} \left| \left \langle f_2, \Phi_{P_{m+2k} \left(2^{-k}r+a+2^{\frac{m}{3}}s^{\frac{2}{3}} \right)} \right \rangle \right| \\
&& \quad \quad \quad \quad \quad \quad \cdot \left| \left\langle f_1, \Phi_{P_{m+k} \left(2^{-2k}r+2^{-k}a+b-2^{\frac{2m}{3}}s^{\frac{1}{3}} \right)} \right \rangle \right|\left| \left\langle f_4, \Phi_{P_{m+k} \left(2^{-2k}r+2^{-k}a+b\right)} \right \rangle \right| \\
&& \lesssim \frac{1}{2^m} \cdot \frac{1}{2^{m+k}} \sum_{\substack{s \sim 2^m \\ r \in \Z}} \sum_{a, b \sim 2^{-k}} \cdot \frac{\int_{I_{r+s}^{m+3k}}|f_3|}{\left|I_0^{m+3k} \right|} \frac{\int_{I_{2^{-k}r+a+2^{\frac{m}{3}}s^{\frac{2}{3}}}^{m+2k}}|f_2|}{\left|I_0^{m+2k} \right|} \cdot \frac{\int_{I_{2^{-2k}r+2^{-k}a+b-2^{\frac{2m}{3}}s^{\frac{1}{3}}}^{m+k}}|f_1|}{\left|I_0^{m+k} \right|} \cdot \frac{\int_{I_{2^{-2k}r+2^{-k}a+b}^{m+k}}|f_4|}{\left|I_0^{m+k} \right|}.
\end{eqnarray*}
Following now a similar argument as in Section \ref{20231112sec01}, we see that 
$$
\Lambda_m^{k}(\vec{f}) \lesssim |F_1||F_3||F_4|^{-1},
$$
and therefore, by symmetry, we deduce that 
$$
\Lambda_m^{k}(\vec{f}) \lesssim |F_1|^{\theta_1}|F_2|^{\theta_2} |F_3|^{\theta_3}|F_4|^{-1}, 
$$
for $0 \le \theta_1, \theta_2, \theta_3 \le 1$ with $\theta_1+\theta_2+\theta_3=2$.

\medskip 

\subsection{Treatment of the low oscillatory component} \label{20231112sec01}
Recall that 
$$
\Lambda_{j, l, m}^k(\vec{f}):= \iint_{\R^2} \left(f_1*\check{\phi}_{j+k} \right)(x-t) \left(f_2*\check{\phi}_{l+2k} \right)(x+t^2)\left(f_3*\check{\phi}_{m+3k} \right)(x+t^3)f_4(x) 2^k \rho(2^k t)dtdx,
$$
and that the \emph{low oscillatory component} is given by
$$
\Lambda^{Lo}(\vec{f}):=\sum_{k \in \Z} \sum_{(j, l, m) \in \Z^3 \backslash \N^3}\Lambda_{j, l, m}^k(\vec{f}). 
$$
In this section, we employ the strategy presented in \cite[Section 10]{HL23} in order to prove the following

\begin{prop} \label{Mainproplow} 
Let $\{F_j\}_{j=1}^{4}$ be any real measurable sets of finite measure. Then, there exists $F_4'\subseteq F_4$ major subset such that for any $|f_j|\leq \chi_{F_j}$ with $1\leq j\leq 3$ and $|f_4|\leq \chi_{F'_4}$ and any $\vec{\theta}=(\theta_1, \theta_2, \theta_3)$ with $0<\theta_j<1$ the following holds:
\begin{equation}\label{diaglargeboundlo}
\left|\Lambda^{Lo}(\vec{f})\right|\lesssim_{\vec{\theta}} |F_1|^{\theta_1}\,|F_2|^{\theta_2}\,|F_3|^{\theta_3}\,|F_4|^{1-\theta_1-\theta_2-\theta_3}\,.
\end{equation}
\end{prop}

The proof of this result is split into three cases:

\subsubsection{Case 1: All the three indices $j, l, m$ are negative.} \label{20231116subsec30}
Recall that by \cite[(10.2)]{HL23}, it suffices to consider the expression
\begin{eqnarray} \label{20231116eq40}
&& \Lambda_{-, -, -}(\vec{f})= \sum_{l_1, l_2, l_3 \ge 0} \Lambda_{-, -, -}^{l_1, l_2, l_3}(\vec{f}) \nonumber \\
&&\quad := \sum_{\substack{l_1, l_2, l_3 \ge 0 \\ l_1+l_3 \ \textrm{odd}}} \frac{C_{l_1, l_2, l_3}}{l_1!l_2!l_3!} \sum_{k \in \Z} \int_{\R} f_{1, l_1, k}(x)f_{2, l_2, 2k}(x)f_{3, l_3, 3k}(x)f_4(x) \left(\int_{\R} t^{l_1+2l_2+3l_3}\rho(t)dt \right) dx,
\end{eqnarray}
where $f_{i, l_i, ik}:=f_i* \left(\Psi_{i}^{l_i} \right)_{ik}$ with  $\Psi_{i}^{l_i}:=\left( \left(\cdot \right)^{l_i}\phi_{-}(\cdot) \right)^{\vee}$, $\Psi_k(x):=2^k \Psi(2^k x)$ and $\phi_{-}$ an even smooth function with $\supp \ \phi_{-} \subseteq \left[-5, 5 \right]$. 
\begin{proof}

We further consider two sub-cases.

\medskip 

\noindent\fbox{\textsf{Case 1.1: $(l_1, l_2, l_3) \neq (a, 0, 0), (0, 0, b), \ a, b \in \Z_+$}.} In this case, at least two of the $l$'s in $\{l_1, l_2, l_3\}$ take positive values. Without loss of generality, we assume\footnote{If $l_2, l_3>0$, then one can modify the proof accordingly by interchanging the role of $l_1$ and $l_3$.} $l_1, l_2>0$, and define the exceptional set $\Omega$ by 
\begin{equation} \label{20231105eq01}
\Omega:=\left\{\calS \one_{F_1} \ge 100\frac{\left|F_1\right|}{\left|F_4\right|} \right\} \cup \left\{\calS \one_{F_2} \ge 100\frac{\left|F_2\right|}{\left|F_4\right|} \right\} \cup \left\{M\one_{F_3} \ge 100 \frac{|F_3|}{|F_4|} \right\}
\end{equation} 
and the function $f_4=\one_{F_4'}$ with $F_4':=F_4 \backslash \Omega$ which, for $C>0$ suitable chosen, implies  $\frac{|F_4|}{2} \le |F_4'| \le |F_4|$. Now, from the argument in \cite[Section 10.1]{HL23}, we have 
\begin{eqnarray} \label{20231105eq02}
 \left| \sum_{\substack{l_1, l_2, l_3 \ge 0 \\ (l_1, l_2, l_3) \neq (1, 0, 0), (0, 0, 1)}}  \Lambda_{-, -, -}^{l_1, l_2, l_3} (\vec{f}) \right| %
 &\lesssim&  \sum_{l_1, l_2, l_3} \frac{C_{l_1, l_2, l_3}}{l_1!l_2!l_3!} \int_{\R} \calS f_1(x) \calS f_2(x) Mf_3(x) |f_4(x)| dx \nonumber \\
&\lesssim& \int_{F'_4
} \calS f_1(x) \calS f_2(x) Mf_3(x) dx.
\end{eqnarray}
Now, using the definition of the exception set $\Omega$, from \eqref{20231105eq01} we deduce
$$
\eqref{20231105eq02} \lesssim \frac{|F_1||F_2||F_3|}{|F_4|^3} \cdot |F_4'| \lesssim |F_1||F_2||F_3||F_4|^{-2}. 
$$

\medskip 

\noindent\fbox{\textsf{Case 1.2: $(l_1, l_2, l_3)=(a, 0, 0) \ \textrm{or} \ (0, 0, b), \ a, b \in \Z_+$}.} Due to the symmetry,  it is enough to consider the case when $(l_1, l_2, l_3)=(1, 0, 0)$. In this situation, we have the natural decomposition 
$$
\left|\Lambda_{-, -, -}^{1, 0, 0}(\vec{f}) \right| \le \left|\Lambda_{-, -, -; +}^{1, 0, 0}(\vec{f}) \right|+\left|\Lambda_{-, -, -; -}^{1, 0, 0}(\vec{f}) \right|,
$$
where in the right-hand side above the first term refers to the summation $\sum_{k \in \N}$, while the second term corresponds to the summation $\sum_{k \in \Z_{-}}$. 

\medskip 

\noindent\underline{\textsf{Treatment of $\Lambda_{-, -, -; -}^{1, 0, 0}(\vec{f})$.}} In this case, we define
\begin{equation} \label{20231105eq03}
\Omega:=\left\{\calS \one_{F_1} \ge 100 \frac{|F_1|}{|F_4|} \right\} \cup  \left\{M\one_{F_2} \ge 100\frac{\left|F_2\right|}{\left|F_4\right|} \right\} \cup \left\{M\one_{F_3} \ge 100 \frac{|F_3|}{|F_4|} \right\}\,,
\end{equation} 
and notice that
\begin{equation} \label{20231105eq04}
\left|\Lambda_{-, -, -; -}^{1, 0, 0}(\vec{f}) \right| \lesssim \int_{\R} \calS f_1(x) Mf_2(x)Mf_3(x) \calS f_4(x) dx.
\end{equation}
Using by now standard reasonings in the vein of the discussion in Section \ref{MaindiaglargeK} with special stress on the type of reasonings in \textsf{Step VII} (see in particular \eqref{20231027eq01*}--\eqref{20231027eq02*}) one gets 
$$
\eqref{20231105eq04}  \lesssim |F_1|^{\frac{2+\theta}{3}}|F_2|^{\frac{2+\theta}{3}}|F_3|^{\frac{2+\theta}{3}}|F_4|^{-1-\theta}\,,\qquad  \forall \:\theta \in (0, 1).
$$

\medskip 

\noindent\underline{\textsf{Treatment of $\Lambda_{-, -, -; +}^{1, 0, 0}(\vec{f})$.}} In this case, for any $\epsilon>0$ sufficiently small and $C>0$ large enough we set
\begin{eqnarray} \label{20231105eq05}
\Omega_\epsilon%
&:=& \left\{\calP_{\frakM}(f_1, f_2) \ge C \frac{|F_1|^{\frac{1}{1+\epsilon}}}{|F_4|^{\frac{1}{1+\epsilon}}} \right\} \cup \left\{\calP_{\frakM}(f_1, f_2) \ge C \frac{|F_2|^{\frac{1}{1+\epsilon}}}{|F_4|^{\frac{1}{1+\epsilon}}} \right\} \nonumber \\
&& \cup \left\{\calP_{\frakM}(f_1, f_2) \ge C\frac{|F_1|^{\frac{1}{1+\epsilon}}|F_2|^{\frac{1}{1+\epsilon}}}{|F_4|^{\frac{2}{1+\epsilon}}} \right\} \cup \left\{\calS \one_{F_3} \ge C \frac{|F_3|}{|F_4|} \right\},
\end{eqnarray}
where, for $\psi \in C_0^{\infty}(\R)$ with $\supp \ \psi \subseteq \left\{|\xi|<1 \right\}$, the expression
$$
\calP_{\frakM}(f_1, f_2):=\sup_{N \in \N} \left| \sum_{k=0}^N \left(f_1*\check{\psi}_k \right)(x) \left(f_2* \check{\phi}_{2k} \right) (x) \right|
$$
is the \emph{maximal (anisotropic) paraproduct operator}. It is known that $\calP_{\frakM}: L^{p_1}(\R) \times L^{p_2}(\R) \mapsto L^r(\R)$ for any $\frac{1}{p_1}+\frac{1}{p_2}=\frac{1}{r}$ with $1<p_1, p_2 \le \infty$ and $\frac{1}{2}<r<\infty$ (see, e.g., \cite[Section 10.1]{HL23}). As usual, we let $f_4:=\one_{F_4'}$ where $F_4':=F_4 \backslash \Omega_\epsilon$ with $\frac{F_4}{2} \le |F_4'| \le |F_4|$. 

Recall \cite[Section 10.1]{HL23} we notice that in this case it is sufficient to control the term 
\begin{equation} \label{20231105eq06}
\int_{\R} \calP_{\frakM}(f_1, f_2)(x) \calS f_3(x) \calS f_4(x) dx.
\end{equation} 
Again, applying a level set decomposition in the spirit of  Section \ref{MaindiaglargeK}  one concludes
\begin{equation}\label{concl}
\eqref{20231105eq06} \lesssim_{\ep,\theta} \frac{|F_1|^{\frac{1}{1+\epsilon}}|F_2|^{\frac{1}{1+\epsilon}}}{|F_4|^{\frac{2}{1+\epsilon}}} |F_3|^{\theta} |F'_4|^{1-\theta} \lesssim |F_1|^{\frac{1}{1+\epsilon}}|F_2|^{\frac{1}{1+\epsilon}} |F_3|^{\theta} |F_4|^{1-\theta-\frac{2}{1+\epsilon}}, \quad \forall \epsilon, \theta \in (0, 1). 
\end{equation}

\subsubsection{Case 2: Two of the indices are negative and one is positive} \label{20231116subsec02a}
Without loss of generality, we may assume that $j, l<0$ and $m \ge 0$ as the other two cases can be treated similarly. With this, we have
$$
\Lambda_{-, -, +}(\vec{f}) \simeq \sum_{m \ge 0} \frac{1}{2^m} \sum_{\substack{l_1, l_2\geq 0\\(l_1,l_2)\not=(0,0)}} \sum_{k \ge \Z} \int_{\R} f_{1, l_1, k}(x) f_{2, l_2, 2k} (x) \left(f_3*\check{\phi}_{m+3k} \right)(x) f_4(x) dx. 
$$
Since we must have $(l_1,l_2)\not=(0,0)$ it is enough to consider only two cases:

\medskip 

\textit{$\bullet$ If $l_1>0$ and $l_2 > 0$}, then by \cite[Section 10.2]{HL23}, one has
$$
\left|\Lambda_{-, -, +}(\vec{f}) \right| \lesssim \sum_{m \ge 0} \:\sum_{l_1, l_2 > 0} \frac{1}{2^m}  \frac{\left|C_{l_1, l_2} \right|}{l_1! l_2!} \int_{\R} \calS f_1(x) Mf_2(x) \calS f_3(x) |f_4(x)| dx. 
$$
and one can follow a similar argument with that in \textsf{Case 1.1} of Section \ref{20231116subsec30} above.

\medskip

\textit{$\bullet$ If $l_1=0$ or $l_2=0$} then one can mimic the arguments provided in \textsf{Case 1.2} of Section \ref{20231116subsec30}.

\subsubsection{Case 3: One index is negative and two are positive} \label{20231116subsec02}

Again, without loss of generality, we may assume $j, l \ge 0$ and $m<0$ as the other two cases can be treated similarly. Recall from \cite[Section 10.3]{HL23} that in this case, it suffices to estimate the term
$$
\sum_{l_3 \ge 0} \frac{|C_{l_3|}}{l_3!} \int_{\R} \left|B_{l_3}(f_1, f_2)(x) \right| Mf_3(x) |f_4(x)| dx, 
$$
where, for $\rho_{l_3}(t):=t^{3l_3} \rho(t)$, we let 
$$
B_{l_3}(f_1, f_2)(x):=\sum_{\substack{j, k \ge 0 \\ k \in \Z}} \int_{\R^2} \hat{f_1}(\xi) \hat{f_2}(\eta) \phi \left(\frac{\xi}{2^{j+k}} \right) \phi \left(\frac{\eta}{2^{l+2k}} \right) \left(\int_{\R} e^{i \left(-\frac{\xi}{2^k}t+\frac{\eta}{2^{2k}} t^2 \right)} \rho_{l_3}(t) dt \right) e^{ix(\xi+\eta)} d\xi d\eta, 
$$
be the stationary component of the \emph{curved bilinear Hilbert transform along the curve} $(-t, \gamma(t))$ as defined in \cite{Lie15} where in the present context\footnote{In the other two cases, $\gamma(t)$ is given by $t^3$ and $t^{\frac{3}{2}}$, respectively.} $\gamma(t)=t^2$. Now, based on \cite{GL22} (see also \cite{LX16}) we know that $B_{l_3}: L^{p_1}(\R) \times L^{p_2}(\R) \to L^r(\R)$ with $\frac{1}{p_1}+\frac{1}{p_2}=\frac{1}{r}$ where $r>\frac{1}{2}$ and $p_1, p_2>1$. Therefore, for any $\epsilon>0$ and $C>0$ absolute constant large enough, we can define the exceptional set $\Omega_{l_3; \epsilon}$ to be
\begin{eqnarray*}
\Omega_{l_3; \epsilon}%
&:=& \left\{B_{l_3}(f_1, f_2) \ge C \frac{|F_1|^{\frac{1}{1+\epsilon}}}{|F_4|^{\frac{1}{1+\epsilon}}} \right\} \cup \left\{B_{l_3}(f_1, f_2) \ge C \frac{|F_2|^{\frac{1}{1+\epsilon}}}{|F_4|^{\frac{1}{1+\epsilon}}} \right\} \nonumber \\
&& \cup \left\{B_{l_3}(f_1, f_2) \ge C\frac{|F_1|^{\frac{1}{1+\epsilon}}|F_2|^{\frac{1}{1+\epsilon}}}{|F_4|^{\frac{2}{1+\epsilon}}} \right\} \cup \left\{\calS \one_{F_3} \ge C \frac{|F_3|}{|F_4|} \right\}.
\end{eqnarray*}
The rest of the reasoning follows from the arguments in \textsf{Case 1.2} of Section \ref{20231116subsec30} in order to conclude that the same bound as in \eqref{concl} holds thus finishing the proof of Proposition \ref{Mainproplow}. 
\end{proof}

\subsection{The extended boundedness range: the quasi-Banach output}
In this section we provide an outline of how to put together all the restricted weak-type estimates we got so far in order to obtain the desired quasi-Banach range bounds for the quadrilinear form $\Lambda$ associated with the curved tri-linear Hilbert transform. Recall now our decomposition
$$
\Lambda=\Lambda^{Hi}+\Lambda^{Lo}=\Lambda_+^{Hi}+\Lambda_{-}^{Hi}+\Lambda^{Lo}. 
$$

\subsubsection{The high oscillatory component $\Lambda_+^{Hi}$}
The main expression we want to control is given by
$$
\Lambda_{j, l, m}^k(\vec{f})=\int_{\R^2} \left(f_1* \check{\phi}_{j+k}\right)\left(x-\frac{t}{2^k} \right) \left(f_2* \check{\phi}_{l+2k}\right)\left(x+\frac{t^2}{2^{2k}} \right) \left(f_3*\check{\phi}_{m+3k}\right)\left(x+\frac{t^3}{2^{3k}} \right) f_4(x) \rho(t)dtdx,
$$
where $\supp \ \widehat{f_4} \subseteq A(k, j, l, m):=\left\{\xi \sim 2^{j+k} \right\}+\left\{\eta \sim 2^{l+2k} \right\}+ \left\{\tau \sim 2^{m+3k} \right\}$.

In what follows, we always assume that 
\begin{enumerate}
    \item [$\bullet$] $F_j, j=1, 2, 3, 4$ are any arbitrary finitely measurable sets in $\R$;
    \item [$\bullet$] $f_j=\one_{F_j}, j=1, 2, 3$ and $f_4=\one_{F_4'}$ where $F_4' \subseteq F_4$ measurable with $\frac{F_4}{2} \le |F_4'| \le |F_4|$. 
\end{enumerate}

Putting together the restricted weak-type estimates for the $\Lambda_+^{Hi}$ components from Sections \ref{MaindiaglargeK}---\ref{secoffdiag}, we have
\begin{enumerate}
\item [(1).] for any $k \in \N$
and $0< \theta_1, \theta_2, \theta_3<1$ with $\theta_1+\theta_2+\theta_3\leq 2$, one has
\begin{equation} \label{20231112eq20}
\left|\Lambda_{j, l, m}^k \left(\vec{f} \right) \right| \lesssim_{\vec{\theta}} |F_1|^{\theta_1}|F_2|^{\theta_2}|F_3|^{\theta_3}|F_4|^{1-\theta_1-\theta_2-\theta_3}\,;
\end{equation} 

\item [(2).] for any $0< \theta_1, \theta_2, \theta_3<1$ one has
\begin{equation} \label{20231112eq301}
\left| \sum_{k \ge \max\{j, l, m\}} \Lambda_{j, l, m}^k \left(\vec{f} \right) \right| \lesssim_{\vec{\theta}}  \left(\max\{j, l, m\} \right)^{10} |F_1|^{\theta_1}|F_2|^{\theta_2}|F_3|^{\theta_3}|F_4|^{1-\theta_1-\theta_2-\theta_3}\,. 
\end{equation} 
\end{enumerate}

With these done, the desired control on $\Lambda_+^{Hi}$ is obtained as follows:

\begin{itemize}
\item firstly, one applies multi-linear interpolation between \eqref{20231112eq20} and the key local $L^2$ estimate in \cite[Section 12.1]{HL23} that reads \begin{equation} \label{20230323eq01}
\left| \Lambda_{j, l, m}^k (\vec{f}) \right| \lesssim_{\ep} 2^{-\epsilon \max \left\{  \min \left\{ 2|k|, \max\{j, l, m\} \right\}, \max \left\{|j-l|, |l-m|, |j-m| \right\} \right\}} \|\vec{f}\|_{L^{\vec{p}}} \quad \forall\,\vec{p} \in {\bf H}^{+}\,,
\end{equation}
where here ${\bf H}^{+}=\left\{(p_1,p_2,p_3,p_4)\,\Big|\,\sum_{j=1}^4\frac{1}{p_j}=1\:\:\textrm{and}\:\:p_1=2\:\:\&\:\:(p_2,p_3,p_4)\in\{2,\infty\}\right\}$ and $\ep>0$ is a properly chosen small constant.

\item secondly,  one applies multi-linear interpolation between \eqref{20231112eq301} and the key local $L^2$ estimate in \cite[Section 12.1]{HL23} given by 
\begin{equation} \label{20230323eq02}
\left| \sum_{k \ge \max\{j, l, m\}} \Lambda_{j, l, m}^k \left(\vec{f} \right) \right|  \lesssim_{\ep} 2^{-\epsilon \max\{j, l, m\}} \|\vec{f}\|_{L^{\vec{p}}} \quad \forall\,\vec{p} \in {\bf H}^{+}\,.
\end{equation}

\item finally, one follows similar reasonings with the ones in  \cite[Section 12.1]{HL23} (more precisely, the decomposition (12.4)--(12.6) therein) in order to conclude that 
\begin{equation} \label{20231112eq21}
\Lambda_{+}^{Hi}: L^{p_1}(\R) \times L^{p_2}(\R) \times L^{p_3}(\R) \longrightarrow L^{r, \infty}(\R),
\end{equation} 
with $\frac{1}{p_1}+\frac{1}{p_2}+\frac{1}{p_3}=\frac{1}{r}$, where $1<p_1, p_2, p_3<\infty$ and $\frac{1}{2}<r<\infty$. 
\end{itemize}

\subsubsection{The high oscillatory component $\Lambda_-^{Hi}$}

Recall that the only difference appearing in this case, is that the roles of $F_1$ and $F_3$ are interchanged. Consequently, the main estimates \eqref{20231112eq20} and \eqref{20231112eq301} (with the obvious adaptations) remain valid and thus the analogue of the main conclusion \eqref{20231112eq21} holds. 

\subsubsection{The low oscillatory component $\Lambda^{Lo}$}
Recalling Proposition \ref{Mainproplow} in Section \ref{20231112sec01} and applying multilinear interpolation we immediately deduce that
$$
\Lambda_{Lo}: L^{p_1}(\R) \times L^{p_2}(\R) \times L^{p_3}(\R) \longmapsto L^{r, \infty}(\R), 
$$
with $\frac{1}{p_1}+\frac{1}{p_2}+\frac{1}{p_3}=\frac{1}{r}$, where $1<p_1, p_2, p_3<\infty$ and $\frac{1}{3}<r<\infty$. 

\section{The curved $n$-sublinear maximal operator}

In this section we provide the short proof of Theorem \ref{mainthm03}. The key difference between the curved multi-linear Hilbert transform and the curved $n$-sublinear maximal operator is that in the latter case, we have a positive kernel and thus, in particular, $\calM_{n,\vec{\al},\vec{\be}} (\vec{f})$ features the key \emph{monotonicity property}:
\begin{equation} \label{monot}
\textrm{if}\:\:\:|f_j|\leq |g_j|\:\quad\:\forall\:\:1\leq j\leq n\:\quad\textrm{then}\quad \calM_{n,\vec{\al},\vec{\be}} (\vec{f})\leq \calM_{n,\vec{\al},\vec{\be}} (\vec{g})\,.
\end{equation} 
As a consequence, the proof of Theorem \ref{mainthm03} can be easily recovered from the known $L^p$-bounds satisfied by the curved bilinear maximal operator. 
\medskip

\noindent\textit{Proof of Theorem \ref{mainthm03}.}
For expository reasons and notational simplicity we take $\vec{\be}=\vec{\one}:=(1,\ldots,1)$ and set $\calM_{n,\vec{\al},\vec{\one}}= \calM_{n,\vec{\al}}$. With these we immediately notice that
\begin{equation} \label{20231118eq02}
\calM_{n,\vec{\al}}(\vec{f})(x) \simeq \widetilde{\calM}_{n,\vec{\al}}(\vec{f})(x), \quad a.e. \ x \in \R, 
\end{equation} 
where 
$$
\widetilde{\calM}_{n,\vec{\al}}(\vec{f})(x):=\sup_{\epsilon>0} \frac{1}{2\epsilon} \int_{\frac{\epsilon}{2}<|t|<\epsilon}  \left|f_1\left(x+t^{\alpha_1}\right)f_2\left(x+t^{\alpha_2}\right) \cdots f_n \left(x+t^{\alpha_n}\right) \right| dt, \quad x \in \R.
$$
Indeed, we trivially have $\widetilde{\calM}_{n,\vec{\al}} (\vec{f})(x) \lesssim \calM_{n,\vec{\al}}(\vec{f})(x)$ while for the reverse direction, fixing an $\epsilon>0$, we have
\begin{eqnarray*}
&&  \frac{1}{2\epsilon} \int_{-\epsilon}^{\epsilon}  \left|f_1\left(x+t^{\alpha_1}\right)f_2\left(x+t^{\alpha_2}\right) \cdots f_n \left(x+t^{\alpha_n}\right) \right| dt \\
&& \le \sum_{j=0}^\infty \frac{1}{2^{j}} \cdot \frac{2^j}{2\epsilon} \int_{\frac{\epsilon}{2^{j+1}}<|t|<\frac{\epsilon}{2^j}} \left|f_1\left(x+t^{\alpha_1}\right)f_2\left(x+t^{\alpha_2}\right) \cdots f_n \left(x+t^{\alpha_n}\right) \right| dt \\
&& \le \left(\sum_{j=0}^\infty \frac{1}{2^{j+1}} \right) \cdot \widetilde{\calM}_{n,\vec{\al}}(\vec{f})(x) \lesssim \widetilde{\calM}_{n,\vec{\al}}(\vec{f})(x).
\end{eqnarray*}

Next,  we prove Theorem \ref{CnLMO} under the assumption that precisely two $p_i$'s are different from $\infty$: without loss of generality, we consider the case
$$
1<p_1, p_2 \le \infty \quad  \textrm{and} \quad  p_3= \dots p_n=\infty. 
$$
In this situation though, based on \eqref{monot}, it is sufficient to estimate the (sub)bilinear maximal operator 
$$
\widetilde{\calM}_{2, (\alpha_1, \alpha_2)} (f_1, f_2)(x):= \sup_{\epsilon>0} \frac{1}{2\epsilon} \int_{\frac{\epsilon}{2} \le |t|<\epsilon} \left|f_1(x+t^{\alpha_1}) f_2(x+t^{\alpha_2}) \right|dt.
$$
Applying the change variable $s=t^{\alpha_1}$, we see that
\begin{eqnarray} \label{20230904eq01}
\widetilde{\calM}_{2, (\alpha_1, \alpha_2)} (f_1, f_2)(x)%
&\simeq&  \sup_{\epsilon>0} \frac{1}{\epsilon^{\alpha_1}} \int_{\frac{\epsilon^{\alpha_1}}{2^{\alpha_1}}<|s|<\epsilon^{\alpha_1}} \left|f_1(x+t)f_2 \left(x+t^{\frac{\alpha_2}{\alpha_1}} \right) \right| ds \nonumber \\
&=& \sup_{\epsilon>0} \frac{1}{\epsilon} \int_{\frac{\epsilon}{2^{\alpha_1}}<|s|<\epsilon} \left|f_1(x+t)f_2 \left(x+t^{\frac{\alpha_2}{\alpha_1}} \right) \right| ds. 
\end{eqnarray}
Observe that the curve $\gamma(t):=t^{\frac{\alpha_2}{\alpha_1}}$ belongs to the class $\calN\calF$, and hence by \cite[Theorem 3]{LX16}, the operator \eqref{20230904eq01} extends boundedly from $L^{p_1}(\R) \times L^{p_2}(\R)$ into $L^r(\R)$ for any $p_1,\,p_2,\,r$ obeying 
$$
\frac{1}{p_1}+\frac{1}{p_2}=\frac{1}{r} \quad \textrm{with} \quad 1<p_1, p_2 \le \infty \quad \textrm{and} \quad \frac{1}{2}< r \le \infty. 
$$
Finally, the boundedness of $\calM_{n,\vec{\al}}$ in the general quasi-Banach range stated in Theorem \ref{mainthm03}
follows from multi-linear interpolation. 
\hfill\qedsymbol{}

\begin{rem}
It is easy to notice that the boundedness range of $\calM_{n,\vec{\al}}$ does not contain the endpoints of the form $(p_1,\ldots,p_n, r)$ of the form $p_{j_0}=r=1$ and $p_j=\infty$ for $j\not=j_0$. This follows from a straightforward adaptation of the example provided in \cite[Section 5.3]{GL20}.
\end{rem}



\begin{thebibliography}{10}

\bibitem{BK25}
Lars Becker and Ben Krause.
\newblock On {M}ulti-linear {M}aximal {O}perators Along Homogeneous Curves.
\newblock Arxiv: https://arxiv.org/abs/2508.09080, 10pp, 2025.

\bibitem{BBL25}
Cristina Benea, Frederic Bernicot, and Victor Lie.
\newblock On the hybrid trilinear {H}ilbert transform.
\newblock Preprint, 2025.

\bibitem{BHL25}
B\'enyi Arp\'ad, Bingyang Hu, and Victor Lie.
\newblock The {B}ilinear {H}ilbert--{C}arleson operator along curves. {T}he
  purely non-zero curvature case.
\newblock Arxiv: https://arxiv.org/abs/2507.04467, 59 pp, 2025, submitted.

\bibitem{BL96}
V.~Bergelson and A.~Leibman.
\newblock Polynomial extensions of van der {W}aerden's and {S}zemer\'{e}di's
  theorems.
\newblock {\em J. Amer. Math. Soc.}, 9(3):725--753, 1996.

\bibitem{Cal}
A.-P. Calder\'on.
\newblock Cauchy integrals on {L}ipschitz curves and related operators.
\newblock {\em Proc. Nat. Acad. Sci. U.S.A.}, 74(4):1324--1327, 1977.

\bibitem{Car66}
Lennart Carleson.
\newblock On convergence and growth of partial sums of {F}ourier series.
\newblock {\em Acta Math.}, 116:135--157, 1966.

\bibitem{CMM}
R.~R. Coifman, A.~McIntosh, and Y.~Meyer.
\newblock L'int\'egrale de {C}auchy d\'efinit un op\'erateur born\'e sur
  {$L^{2}$}pour les courbes lipschitziennes.
\newblock {\em Ann. of Math. (2)}, 116(2):361--387, 1982.

\bibitem{DD}
Dong Dong.
\newblock On the {B}ilinear {H}ilbert transform on two polynomials.
\newblock {\em P.A.M.S. (147)}, 10:4245–4258, 2019.

\bibitem{f}
Charles Fefferman.
\newblock Pointwise convergence of {F}ourier series.
\newblock {\em Ann. of Math. (2)}, 98:551--571, 1973.

\bibitem{GL20}
Alejandra Gaitan and Victor Lie.
\newblock The boundedness of the (sub)bilinear maximal function along
  ``non-flat'' smooth curves.
\newblock {\em J. Fourier Anal. Appl.}, 26(4):Paper No. 69, 33, 2020.

\bibitem{GL22}
Alejandra Gaitan and Victor Lie.
\newblock Non-zero to zero curvature transition: Operators along hybrid curves
  with no quadratic (quasi)resonances.
\newblock {\em Adv. Math.}, 478:123, 2025.

\bibitem{HL30}
Hardy, G. H. and Littlewood, J. E. .
\newblock {\em A maximal theorem with function-theoretic applications}.
\newblock {\em Acta Mathematica}, 54(1), 81–116.

\bibitem{HL23}
Bingyang Hu and Victor Lie.
\newblock On the curved trilinear {H}ilbert transform.
\newblock Arxiv: https://arxiv.org/abs/2308.10706, 103pp, 2023, submitted.

\bibitem{KMPW2024b}
Dariusz Kosz, Mariusz Mirek, Sarah Peluse, and James Wright.
\newblock The multilinear circle method and a question of {Bergelson}.
\newblock {\em arXiv preprint arXiv:2411.09478v2}, 2024.

\bibitem{KMPW2024a}
Ben Krause, Mariusz Mirek, Sarah Peluse, and James Wright.
\newblock Polynomial progressions in topological fields.
\newblock {\em Forum Math. Sigma}, 12:Paper No. e106, 51, 2024.

\bibitem{lt1}
Michael Lacey and Christoph Thiele.
\newblock {$L\sp p$} estimates on the bilinear {H}ilbert transform for
  {$2<p<\infty$}.
\newblock {\em Ann. of Math. (2)}, 146(3):693--724, 1997.

\bibitem{lt2}
Michael Lacey and Christoph Thiele.
\newblock On {C}alder\'on's conjecture.
\newblock {\em Ann. of Math. (2)}, 149(2):475--496, 1999.

\bibitem{la3}
Michael Lacey.
\newblock The bilinear maximal functions map into {$L^p$} for
              {$2/3<p\leq1$}.
\newblock {\em Ann. of Math. (2)}, 151(1):35--57, 2000.

\bibitem{Li13}
Xiaochun Li.
\newblock Bilinear {H}ilbert transforms along curves {I}: {T}he monomial case.
\newblock {\em Anal. PDE}, 6(1):197--220, 2013.

\bibitem{LX16}
Xiaochun Li and Lechao Xiao.
\newblock Uniform estimates for bilinear {H}ilbert transforms and bilinear
  maximal functions associated to polynomials.
\newblock {\em Amer. J. Math.}, 138(4):907--962, 2016.

\bibitem{Lie15}
Victor Lie.
\newblock On the boundedness of the bilinear {H}ilbert transform along
  ``non-flat'' smooth curves.
\newblock {\em Amer. J. Math.}, 137(2):313--363, 2015.

\bibitem{Lie18}
Victor Lie.
\newblock On the boundedness of the bilinear {H}ilbert transform along
  ``non-flat'' smooth curves. {T}he {B}anach triangle case {$(L^r,\ 1\leq
  r<\infty)$}.
\newblock {\em Rev. Mat. Iberoam.}, 34(1):331--353, 2018.

\bibitem{Lie2024}
Victor Lie.
\newblock A unified approach to three themes in harmonic analysis
  ({I}$\,\&\,${II}): ({I}) {T}he linear {H}ilbert {T}ransform and {M}aximal
  {O}perator along variable curves; ({II}) {C}arleson {T}ype operators in the
  presence of curvature.
\newblock {\em Adv. Math.}, 437:113, 2024.

\bibitem{MS13}
Camil Muscalu and Wilhelm Schlag.
\newblock {\em Classical and multilinear harmonic analysis. {V}ol. {II}},
  volume 138 of {\em Cambridge Studies in Advanced Mathematics}.
\newblock Cambridge University Press, Cambridge, 2013.

\bibitem{MR28}
Marcel Riesz.
\newblock {\em Sur les fonctions conjugu\'ees}.
\newblock {\em Mathematische Zeitschrift}, 27:218–244, 1928.

\end{thebibliography}
\end{document}